\documentclass{article}
\usepackage{pgfplots}
\usepackage{fullpage}
\usepackage{amsmath}
\usepackage{amsthm}
\usepackage{amssymb}
\usepackage{url}
\usepackage{tikz-cd}
\usepackage{mathtools}
\usepackage{amsfonts}
\usepackage{graphicx}
\usepackage{verbatim}
\usepackage{url}
\usepackage{graphicx}
\usepackage{siunitx}
\usepackage[caption=false]{subfig}
\usepackage{float}
\usepackage[ruled]{algorithm2e}
\usepackage{tikz}
\usetikzlibrary{calc}
\usetikzlibrary{patterns}
\usepackage{tikz}
\tikzstyle{mybox} = [draw=black, fill=white,  thick,
rectangle, inner sep=10pt, inner ysep=20pt]
\tikzstyle{mybox} = [draw=black, fill=white,  thick,
rectangle, inner sep=2pt, inner ysep=2pt]
\usetikzlibrary{arrows}
\usetikzlibrary{arrows,chains,matrix,positioning,scopes}
\usetikzlibrary{arrows,shapes,positioning}
\usetikzlibrary{decorations.markings}
\tikzstyle arrowstyle=[scale=1]
\tikzstyle directed=[postaction={decorate,decoration={markings,
		mark=at position .65 with {\arrow[arrowstyle]{stealth}}}}]
\tikzstyle reverse directed=[postaction={decorate,decoration={markings,
		mark=at position .65 with {\arrowreversed[arrowstyle]{stealth};}}}]
\newcommand{\boundellipse}[3]
{(#1) ellipse (#2 and #3)
}
\newtheorem{theorem}{Theorem}

\newtheorem{cor}{Corollary}
\newtheorem{prop}{Proposition}
\theoremstyle{definition}
\newtheorem{definition}{Definition}
\newtheorem{remark}{Remark}
\newtheorem{example}{Example}

\usepackage{multirow}
\usepackage[utf8]{inputenc}

\begin{document}

\title{On the Performance of a Novel Class of Linear System Solvers and Comparison with State-of-The-Art Algorithms}
\author{Chun Lau\footnote{Department of Computer Science, Rutgers University (larryl@cs.rutgers.edu).} \ and Bahman Kalantari\footnote{Emeritus Professor, Department of Computer Science, Rutgers University 
  (kalantari@cs.rutgers.edu).}}
\date{}

\maketitle

\begin{abstract}
We present a comprehensive computational study of a class of linear system solvers, called {\it Triangle Algorithm} (TA) and {\it Centering Triangle Algorithm} (CTA), developed by Kalantari \cite{kalantari23}. The algorithms compute an approximate solution or minimum-norm solution to $Ax=b$ or $A^TAx=A^Tb$, where $A$ is an $m \times n$ real matrix of arbitrary rank. The algorithms specialize when $A$ is symmetric positive semi-definite. 

We first summarize the description and theoretical properties of TA and CTA from \cite{kalantari23}. Based on these, we give an implementation of the algorithms that is easy-to-use for practitioners, versatile for a wide range of problems, and robust in that our implementation does not necessitate any constraints on $A$. Our implementation accomplishes the tasks of computing approximate solution or minimum-norm solution. Therefore, this work complements existing literature and tools on solving linear systems, where there is no mention of a linear system solver implementation for arbitrary coefficient matrix $A$. 

Next, we make computational comparisons of our implementation with the Matlab implementations of two state-of-the-art algorithms, the popular GMRES for solving square linear systems as well as ``lsqminnorm" for solving rectangular systems. We consider square and rectangular matrices, for $m$ up to $10000$ and $n$ up to $1000000$, encompassing a variety of applications, such as matrices in  Computational Fluid Dynamics, Clement matrices, all 3689 matrices of the well-known Matrix Market and Suite-Sparse collection \cite{matrixmarket, suitesparse}, and matrices in linear programming (LP) listed in the Suite-Sparse collection. These results indicate that our implementation outperforms GMRES and ``lsqminnorm" both in runtime and quality of residuals. Moreover, as can be seen from our figures, the relative residuals of CTA decrease considerably faster and more consistently than GMRES. In many cases of singular  matrices, GMRES does not even converge. Yet, our implementation provides high precision approximation, faster than GMRES reports lack of convergence. With respect to ``lsqminnorm", our implementation runs faster, producing better solutions.

Additionally, we present a theoretical study in the dynamics of iterations of residuals in CTA and complement it with revealing visualizations. Lastly, we extend TA for LP feasibility problems, handling non-negativity constraints. Computational results show that our implementation for this extension is on par with those of TA and CTA, suggesting applicability in 
linear programming and related problems. 

\end{abstract}





\newpage

\tableofcontents

\newpage 

\section{Introduction}
Gaussian elimination is a direct method  for solving  a linear equation $Ax=b$, a well-known method analyzed in  numerous textbooks,  e.g. Atkinson \cite{Atkinson}, Bini and Pan \cite{Bini}, Golub and van Loan \cite{Golub}, and Strang \cite{Strang}. While in theory Gaussian method and LU factorization  provide means for computing solution  to a solvable system, the method provides no information until the very end. Also, for large dimensions the method is not practical. Iterative methods, generally applicable for solving a square invertible system, generate a sequence of approximate solutions that ideally converge to the solution of the system. In contrast to direct methods, iterative methods offer very important and practical  alternatives, especially when solving  very large or sparse linear systems, such as problems from discretized partial differential equations, where direct methods become prohibitive.
When $A$ is an $m \times n$ matrix, the linear system could be {\it underdetermined} ($m  \leq n$), {\it overdetermined} ($n \geq m$), {\it consistent} or {\it inconsistent}. Thus solving a linear system $Ax=b$ with an $m \times n$ matrix could be interpreted as finding a solution when the system is consistent.  Since the set of solutions may not be unique it may be desirable to compute the {\it minimum-norm} solution.  When the system is inconsistent, we may wish to find a solution to the normal equation $A^TAx=A^Tb$, or its minimum-norm solution.

Iterative methods for linear systems include classical methods such as the Jacobi, the Gauss-Seidel, the {\it successive over-relaxation} (SOR), the {\it accelerated  over-relaxation} (AOR), and the {\it symmetric successive over-relaxation} method which applies when $A$ is symmetric. When $A$ is  symmetric and positive definite, the {\it steepest descent} method (SD) and the {\it conjugate gradient} method (CG) are applicable. Convergence rate of iterative methods can often be substantially accelerated by preconditioning. However, there is a trade-off in computational time. While preconditioning may also help the accuracy in computation of approximation, if the condition number of the underlying matrix is not too large an iterative method may need no preconditioning. Some major articles on preconditioning include   Barrett et al. \cite{Bar}, Golub and van Loan \cite{Golub},  Greenbaum \cite{Green}, Hadjidimos \cite{Hadjid}, Liesen and Strakos \cite{Liesen}, Saad \cite{Saad}, Saad  and Schultz \cite{saad1986gmres}, Simoncini and Szyld \cite{Szyld}, van der Vorst \cite{van1}, Varga \cite{Varga}, and Young \cite{Young}, Braatz and Morari \cite{BRAATZ94} and Wathen \cite{wathenprecon} that includes a vast set of references on the subject.

To guarantee convergence, iterative methods often require the coefficient matrix to satisfy certain conditions, such as invertibility, positive definiteness, diagonal dominance. For a linear system $Ax = b$ with $A$  symmetric positive definite, the Steepest Descent and the Conjugate Gradient methods are among the well-studied iterative methods.  The major computational effort in each iteration of the iterative methods involves matrix-vector multiplication, thus requiring $O(N)$ operations, where $N$ is the number of nonzero entries in the matrix. This makes iterative methods very attractive for solving large sparse systems and also for parallelization, see Demmel \cite{Dem}, Dongarra et al. \cite{Don}, Duff and van der Vorst \cite{Duff},  van der Vorst  \cite{van1}, van der Vorst and Chan \cite{Van2}. When $A$ is only invertible but nonsymmetric, the development of iterative methods is much more complicated than the symmetric positive semidefinite cases. One of the popular iterative methods for solving a general (nonsymmetric) linear system with $A$ invertible is the {\it bi-conjugate gradient method stabilized} (BiCGSTAB), developed by H.A. van der Vorst \cite{Vorst1}. It is a Krylov subspace method that has faster and smoother convergence than {\it conjugate gradient squared method} (CGS). A second popular method for solving invertible linear systems with nonsymmetric matrix is the {\it generalized minimal residual method} (GMRES), developed by Saad  and Schultz \cite{saad1986gmres}. GMRES is a generalization of MINRES method developed by Paige and Saunders \cite{Paige75}. While MINRES applies when the matrix $A$ is symmetric indefinite and invertible, GMRES only requires the invertibility of $A$.  For  details on these iterative algorithms and more see Chen \cite{Chen}, Liesen and Strakos \cite{Liesen}, van der Vorst \cite{Vorst1}, Saad \cite{Saad}, Saad  and Schultz \cite{saad1986gmres}, Simoncini and Szyld \cite{Szyld}.

In this article we consider the computational performance of a class of iterative methods, developed by Kalantari \cite{kalantari23}, intended to solve either the linear system $Ax=b$ or the normal equation $A^TAx=A^Tb$, where $A$ is an $m \times n$ real matrix  of arbitrary rank. The methods use any member of a set of  iteration functions, $\{F_1, \dots, F_m\}$, referred as {\it Centering Triangle Algorithm} (CTA) family.  The methods are motivated by the {\it Triangle Algorithm}, introduced in  Kalantari \cite{kalfull, kalcon}, a method for solving the {\it convex hull membership problem} (CHMP). CHMP is the problem of testing if a given point in a Euclidean space lies in the convex hull of a given compact subset in that Euclidean space. Even when the given set consists of a finite number of points the Triangle Algorithm and its extensions in Awasthi et al. \cite{AKZ} and Kalantari and Zhang \cite{KalY22} find applications in  problems arising in optimization and linear programming, computational geometry, and machine learning. As far as this article is concerned, there is no need to describe the Triangle Algorithm for general sets. Rather, the goal of the article is to describe the Triangle Algorithm for linear systems and the CTA class. It is also not necessary to describe the connections between TA and CTA. For there connection the reader can consult \cite{kalantari23}.  Having described these, we then present an extensive computational study of the performance of TA and the CTA class based on the implementation of the methods for different kinds of matrices, including positive semidefinite, positive definite,  tridiagonal, and random sparse matrices in a variety of real world applications such as PDEs, ODEs, and Computational Fluid Dynamics, matrices having dimension up to 10000, and generating approximations  with precision up to $10^{-15}$.

The remaining sections of the article are as follows. In Section \ref{sec2} we give definition of approximation problems to be addressed in the article.
Section \ref{sec3} we describe the Centering Triangle Algorithm (CTA) family of iteration function and summarize their major theoretical properties, including bounds on iteration complexities, taken from \cite{kalantari23}.  In this Section we also present a study the dynamics of the first member of CTA family and complement it with some visualizations. 
In Section \ref{sec4}  we describe the Triangle Algorithm (TA) and summarize its convergence properties from  \cite{kalantari23}. In Section \ref{sec5} we present a comprehensive set of computational results with an implementation of CTA and TA and make comparisons with GMRES for square matrices and "lsqminnorm" for rectangular matrices. In Section \ref{sec6} we extend TA to linear programming feasibility problem and present computational results. Finally, we make concluding remarks,  

\section{Definition of Approximation Solutions} \label{sec2}

Approximate solutions can be defined in different terms. Computationally, one may be  interested in approximations in terms of the distance between the approximate solution and the unique solution,  or when the solution is not unique, the distance from the  approximate solution to the minimum norm solution.  Alternatively, we may be interested in measuring error with respect to the norm of residuals. Here we  give relevant  definitions used in \cite{kalantari23} and in the computations to be presented in this article. 

\begin{definition}
Given  $\varepsilon>0$,\\

1) We say $x_\varepsilon$ is an $\varepsilon$-approximate solution of $Ax=b$ if $\Vert Ax_\varepsilon  - b \Vert \leq \varepsilon$. 

2) We say $x_\varepsilon$ is an $\varepsilon$-approximate solution of the normal equation if $\Vert A^TA x_\varepsilon - A^Tb\Vert \leq \varepsilon$.
\end{definition}

To define approximate minimum-norm solutions, first consider the ellipsoid
\begin{equation}
E_{A,\rho}=\{Ax: \Vert x \Vert \leq \rho\},
\end{equation}
the image of the central ball of radius $\rho$  under the linear mapping defined by $A$. When $A$ is invertible
this is a full-dimensional ellipsoid (see Figure \ref{minnormfig}).

\begin{definition} \label{minnormapp} Given $\varepsilon \in (0,1)$, 
we say $x_\varepsilon$ is an $\varepsilon$-approximate minimum-norm solution of $Ax=b$ if it satisfies the following three conditions:

1. $x_\varepsilon$ satisfies $\Vert Ax_\varepsilon - b \Vert \leq \varepsilon$.

2. There exists a positive number $\underline \rho$ such that $\underline \rho \leq \rho_\varepsilon = \Vert x_\varepsilon \Vert$ and
$b \not \in E^\circ_{A, \underline \rho}= \{Ax: \Vert x \Vert < \underline \rho\} $.

3. $\rho_\varepsilon - \underline \rho \leq \varepsilon$.
\end{definition}

\begin{prop}  \label{propx} Suppose $Ax=b$ is solvable, $x_*$  its minimum-norm solution. If $x_\varepsilon$ is an $\varepsilon$-approximate minimum-norm solution and $\Vert x_* \Vert \leq \Vert x_\varepsilon \Vert$, then $\Vert x_\varepsilon - x_* \Vert \leq \varepsilon$. \qed
\end{prop}

Proposition \ref{propx} shows  the three conditions make $\varepsilon$-approximation to the minimum-norm solution, $x_*$, a meaningful approximation.  Even if the norm of $x_*$ exceeds the norm of  $x_\varepsilon$, computation of $x_\varepsilon$ provides a useful information. In fact even if $Ax=b$ is unsolvable but for a given $\varepsilon$ it is possible to compute $\varepsilon$-approximate solution, this approximate solution provides information in terms of the proximity of the system to feasibility with respect to that precision.

 Condition 1 implies $b_\varepsilon =Ax_\varepsilon$, a boundary point of the ellipsoid $E_{A, \rho_\varepsilon}$, satisfies $\Vert b - b_\varepsilon \Vert \leq \varepsilon$. However, $b$ is not necessarily in  this ellipsoid.  Figure \ref{minnormfig} shows two cases, the outer $b$ is not inside the ellipsoid while the inner $b$ is. Each $b$ is within a distance of $\varepsilon$ from $b_\varepsilon$.  $\varepsilon$-approximate minimum-norm solution to $Ax=b$ is not unique. Condition 2 in particular implies $b$ does not lie in the interior of $E_{A, \rho_\varepsilon}$. Condition 3 implies if $Ax=b$ is solvable the distance between the norm of $x_*$ and the norm $x_\varepsilon$ is within $\varepsilon$.

\begin{figure}[htpb]
	\centering
	\raisebox{10ex}{
	\begin{tikzpicture}[scale=1.1]
	\begin{scope}
	\begin{scope}
	[rotate=47.5829]
	\draw \boundellipse{0,0}{2.81}{1.21};
	\filldraw (0,0) circle (.5pt) node[below] {$O$};
	\end{scope}
\filldraw (-2.5,-3.0)   node[above] {$E_{A, \rho_\varepsilon}$};
\filldraw (-2,-2.4)   node[above] {$E_{A, \underline \rho}$};
	\filldraw (0.2,2.50)  circle (.5pt)  node[above] {$b$};
\filldraw (0.4,2.40)  circle (.5pt)  node[right] {$b_\varepsilon=Ax_\varepsilon$};
\filldraw (0.4,2.20)  circle (.5pt)  node[below] {$b$};
\filldraw (0.6,.8)  circle (.5pt)  node[above] {$x_\varepsilon$};
	\end{scope}
    \begin{scope}
	[rotate=47.5829]
	\draw \boundellipse{0,0}{3.653}{1.573};
	\filldraw (0,0) circle (.5pt) node[below] {$O$};
	\end{scope}
\draw \boundellipse{0,0}{1.0}{1.0};
\draw \boundellipse{0,0}{.75}{.75};
	\end{tikzpicture}}
	\caption{The larger ellipsoid is $E_{A, \rho_\varepsilon}$, the smaller one is $E_{A, \underline \rho}$.  $b_\varepsilon$ approximates $b$.
The figure depicts two scenarios of $b$, one where $b$ is inside $E_{A, \rho_\varepsilon}$ and another where it is outside. The balls of radius $\rho$ and $\Vert x_\varepsilon \Vert$ are also depicted in the figure.}
	\label{minnormfig}
\end{figure}
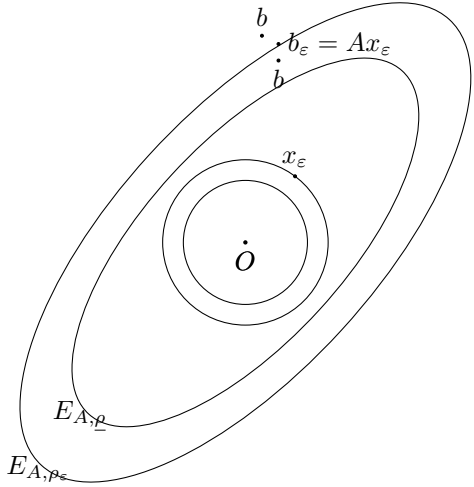

\begin{definition} \label{minnormappnormal}
We say $x_\varepsilon$ is an $\varepsilon$-approximate minimum-norm solution of $A^TAx=A^Tb$ if it satisfies the following conditions:

1. $x_\varepsilon$ satisfies $\Vert A^TAx_\varepsilon - A^Tb \Vert \leq \varepsilon$.

2. There exists a positive number $\underline \rho$ such that $\underline \rho \leq \rho_\varepsilon = \Vert x_\varepsilon \Vert$ and
$A^Tb \not \in E^\circ_{A^TA, \underline \rho}=\{A^TAx: \Vert x \Vert < \underline \rho\}$.

3. $\rho_\varepsilon - \underline \rho \leq \varepsilon$.
\end{definition}

Definition  \ref{minnormappnormal} is simply the same as Definition  \ref{minnormapp} applied to the system $A^TAx=A^Tb$, however the algorithm to be proposed for computing an $\varepsilon$-approximate minimum-norm solution to $A^TAx=A^Tb$ will exploit the fact that $A^TA$ is symmetric PSD.

\begin{definition} Consider the linear system $Ax=b$. Given $x' \in \mathbb{R}^n$, let $b'=Ax'$.  The {\it residual} and {\it least-square residual} are respectively defined as,
\begin{equation}
r'=b-b', \quad A^Tr'.
\end{equation}
\end{definition}

Given $\varepsilon$, the above-mentioned approximate solutions of $Ax=b$ and its normal equation  correspond to computing a residual with norm to within $\varepsilon$, and when this is not possible,
a least-square residual with norm to within $\varepsilon$.

\section{The Centering Triangle Algorithm (CTA)} \label{sec3}
In this section we describe the CTA family and summarize their convergence properties from Kalantari \cite{kalantari23}. The corresponding algorithms
either solve a linear system $Ax=b$ or the {\it normal equation} $A^TAx=A^Tb$, where  $A$ is an $m \times n$ real matrix of arbitrary rank. However, the algorithms exploit the case where $A$ is symmetric positive semidefinite. While $Ax=b$ may be unsolvable, the normal equation is always solvable.   The {\it minimum-norm} solution to $Ax=b$ is a solution that has the minimum norm. The {\it minimum-norm} least-squares solution is the solution to the normal equation  with minimum norm. It will be denoted by $x_*$ and always exists.

\subsection{Technical Foundation of CTA}

Given the $m \times n $ matrix $A$ above, let $H=AA^T$ for a general matrix and $H=A$ when $A$ is symmetric PSD. Given  $t \in \{1, \dots, m \}$, the  {\it Centering Triangle Algorithm} (CTA) of order $t$ refers to the iteration of

\begin{equation}
F_t(r)=r- \sum_{i=1}^t \alpha_{t,i}(r) H^i r,
\end{equation}
where   $\alpha_t(r)=(\alpha_{t,1}(r), \dots, \alpha_{t,t}(r))^T$ is a solution to a convex optimization problem:
\begin{equation}
\alpha_t(r)=
{\rm  argmin} \bigg  \{\Vert r- \sum_{i=1}^t \alpha_{i} H^i r \Vert : (\alpha_1, \alpha_2, \cdots, \alpha_t)^T \in \mathbb{R}^t \bigg \}.
\end{equation}
It can be shown that $\alpha_t(r)$ is  a solution to
the following $t \times t$ linear system, referred as {\it auxiliary equation}
\begin{equation} \label{eqcoeff}
\begin{pmatrix}
\phi_2&\phi_3& \ldots&\phi_{t+1}\\
\phi_3&\phi_4& \ldots&\phi_{t+2}\\
\vdots&\vdots&\ddots&\vdots\\
\phi_{t+1}&\phi_{t+2}& \ldots&\phi_{2t}\\
\end{pmatrix}
\begin{pmatrix}
\alpha_1\\
\alpha_2\\
\vdots\\
\alpha_t\\
\end{pmatrix}
=
\begin{pmatrix}
\phi_1\\
\phi_2\\
\vdots\\
\phi_t\\
\end{pmatrix}, \quad \phi_i=\phi_i(r)=r^TH^ir. 
\end{equation}

It can be shown the auxiliary equation always possess a solution.  When the solution to the auxiliary equation is not unique we may select $\alpha_t(r)$ to be the minimum-norm solution.

Given the current residual $r'=b-Ax'$, the algorithm generates a new residual $r''$ and corresponding approximate solution $x''$, satisfying $r''=b-Ax''$. Note that the residuals and approximate solutions depend on $t$, however for convenience of notation we have suppressed the dependence on $t$. They are defined as:

\begin{equation}  \label{CTAGEN7}
r'' = F_t(r')=r'- \sum_{i=1}^t \alpha_{t,i}(r') H^ir',
\end{equation}
\begin{equation}  \label{CTAGEN7A}
x''= x' +\begin{cases}
\sum_{i=1}^t \alpha_{t,i}(r') H^{i-1}r', & \text{if ~} H=A\\
\sum_{i=1}^t \alpha_{t,i}(r') A^T H^{i-1}r', & \text{if~} H=AA^T.
\end{cases}
\end{equation}

\begin{example} \label{exam1} When $t=1$ and $Hr \not =0$ we get a very simple formula for the iteration of the residuals:
\begin{equation}
F_1(r)=r- \alpha_{1,1}(r)Hr, \quad \alpha_{1,1}(r)= \frac{\phi_1}{\phi_2}= \frac{r^T\!Hr}{r^T\!H^2r}
\end{equation}
Suppose $r_0$ is a residual for $Ax=b$, i.e. $r_0=b-Ax_0$ for some $x_0$.  If $Hr_0 \not =0$, $\Vert F(r_0) \Vert < \Vert r_0 \Vert$. This gives a new residual $r_1$ as $b-Ax_1$. If $Hr_0=0$, then $r_0^THr_0=0$ and hence $A^Tr_0=0$ when $H=AA^T$. When $A$ is symmetric PSD, $Hr_0=0$ implies  $\Vert A^{1/2}r_0 \Vert=0$. But this implies $Ar_0=0$.  Thus if $Hr_0=0$,  $x_0$ is a solution to the normal equation. When $t=2$,  and $\phi_2 \phi_4- \phi_3^2 \not =0$,  from the auxiliary equation we get
\begin{equation} \label{thmeq200xx}
F_2(r) = r - \alpha_{2,1}(r)Hr - \alpha_{2,2}(r) H^2r, 
\end{equation}
where
\begin{equation}
\alpha_{2,1}(r)= \frac{\phi_1 \phi_4-\phi_2 \phi_3}{\phi_2 \phi_4- \phi_3^2}, \quad
\alpha_{2,2}(r)= \frac{\phi^2_2 - \phi_1 \phi_3}{\phi_2 \phi_4- \phi^2_3}.
\end{equation}
\end{example}

\subsection{Description of CTA}
Given a fixed $t \in \{1, \dots, m\}$, Algorithm \ref{1.1} is a conceptually simple algorithm based on computing iterates of $F_t$. It terminates when  either $\Vert r' \Vert \leq \varepsilon$ or 
$\Vert {r'}^T H r' \Vert \leq \varepsilon$. 
If it is known $Ax=b$ is solvable, the while loop of Algorithm \ref{1.1} only needs  the first clause because eventually $\Vert r'  \Vert \leq \varepsilon$.

\begin{algorithm}[!htb]
\SetAlgoNoLine
\KwIn{$A \in \mathbb{R}^{m \times n}$, $b \in \mathbb{R}^m$, Fixed $t \in \{1, \dots, m\}$, $\varepsilon \in (0, 1)$.}
$x' \gets 0$ (or a random point in $\mathbb{R}^n$), $r' \gets b-Ax'$.

\While{$ {\rm (}\|r'\|  > \varepsilon {\rm )} $ $\wedge$ $ {\rm (} r'^T H r'> \varepsilon {\rm )}$}{
$\alpha_t \gets$ min-norm solution to auxiliary equation (see  (\ref{eqcoeff})),

$r' \gets r''$, $x' \gets x''$ (see (\ref{CTAGEN7}) and (\ref{CTAGEN7A}))
}
\caption{(CTA): Computes an $\varepsilon$-approximate solution of $Ax=b$ or $A^TAx=A^Tb$ via iterations of $F_t$, (final $x'$ is approximate solution).} \label{1.1}
\end{algorithm}

An enhanced version of Algorithm \ref{1.1} can be stated when $t \geq 2$. Given $r'=b -Ax'$, the enhanced algorithm stops if either $\Vert r' \Vert \leq \varepsilon$, or 
$\Vert {r'}^T H r' \Vert \leq \varepsilon$, or for some $j \leq t-1$, $\Vert F_1^{\circ (j)}(r')^THF_1^{\circ (j)}(r') \Vert \leq \varepsilon$, where $F^{\circ j}_t(r')= F_t \circ F_t \circ \cdots \circ F_t(r')$, the $j$-fold composition of $F_t$  with itself at $r'$.
In the latter case $j$ is taken to be the smallest such index.  Thus in addition to computing $r''=F_t(r')$ and $x''$ as in (\ref{CTAGEN7A}), the algorithm  may need to compute  $\widehat x''$ as
\begin{equation} \label{widehatxB}
\widehat x''= x' + \begin{cases}
\sum_{i=0}^{j-1} \alpha_{1,1} \big (F_1^{\circ i}(r' )\big ) F_1^{\circ i}(r'), & \text{if ~}
H=A\\
\sum_{i=0}^{j-1} \alpha_{1,1} \big (F_1^{\circ i}(r') \big ) A^T F_1^{\circ i}(r'), & \text{if~} H=AA^T,
\end{cases}
\end{equation}
where $j$ is the smallest index in $\{1, \dots, t-1\}$ such that
$\Vert F^{\circ j}_1(r')^THF^{\circ j}_1(r') \Vert  \leq \varepsilon$ and $\alpha_{1,1}(r)=\phi_1(r)/\phi_2(r)$.\\

\begin{example}
The {\it second-order CTA} is the iterations of  $F_2(r)$, as computed in  Example \ref{exam1}. Given the current residual $r'=b-Ax'$, the next residual $r''$ is defined as

\begin{equation}  \label{CTAx2}
r''= F_2(r')=r' - \alpha_{2,1}(r')Hr' - \alpha_{2,2}(r') H^2r', \quad r''=b-Ax'',
\end{equation}
where
\begin{equation}  \label{CTAx2p}
x''= x' +\begin{cases}
\alpha_{2,1}(r') r' + \alpha_{2,2}(r') Ar', & \text{if ~} H=A\\
\alpha_{2,1}(r') A^Tr' + \alpha_{2,2}(r') A^THr', & \text{if~} H=AA^T.
\end{cases}
\end{equation}
Also,
\begin{equation} \label{CTAx2pp}
\widehat x''= x''+ \begin{cases}
\alpha_{1,1}\big (F_1(r') \big)  F_1(r'), & \text{if ~} H=A\\
\alpha_{1,1}\big (F_1(r') \big) A^TF_1(r'), & \text{if~} H=AA^T.
\end{cases}, \quad \alpha_{1,1}\big (F_1(r') \big)= \frac{F_1(r')^TH F_1(r')}{F_1(r')^TH^2 F_1(r')}.
\end{equation}
The stopping criterion in the algorithm is the satisfiability of at least one of the following three  conditions:
\begin{equation}
\Vert r' \Vert \leq \varepsilon, \quad r'^T H r' \leq \varepsilon, \quad F_1(r')^THF_1(r') \leq \varepsilon.
\end{equation}
\end{example}

\subsection{CTA Convergence Properties} \label{sec2p}

\begin{theorem}  \label{thmone} Given $t \in \{1, \dots, m\}$,  $x_0 \in \mathbb{R}^n$, set $r_0=b -Ax_0$. Denote the sequence of residuals $r''$ and
$x''$ in (\ref{CTAGEN7}) and (\ref{CTAGEN7A}) as $\{r_k\}$ and $\{x_k\}$, respectively.  If ${\rm rank}(A)=m$, $x_k$  converges to the solution of $Ax=b$.  $A^Tb- A^TAx_k$ always converges to zero. If  $\{x_k\}$ is a bounded sequence, any accumulation point, $\overline x$, is a solution to the normal equation. If $Ax=b$ is solvable, $\overline x$ is a solution.
Specifically, given $\varepsilon \in (0,1)$, we have\\

(1) When $H$ is positive definite, $\kappa$ the condition number of $H$,  in $k=O( (\kappa/t) \ln \varepsilon^{-1})$ iterations
\begin{equation} \label{onecon}
\Vert A x_k - b \Vert \leq \varepsilon.
\end{equation}

(2)  When $H$ is singular, $\kappa^+$ the ratio of the largest eigenvalue of $H$ to its smallest positive eigenvalue,  in $k =O(\kappa^+/ \varepsilon)$ iterations, either (\ref{onecon}) holds or the following condition holds:

\begin{equation} \label{twocon}
\Vert A^TA x_{k} - A^Tb \Vert \leq \begin{cases}
\sqrt{\varepsilon} \Vert A \Vert^{1/2} , & \text{if ~} H=A\\
\sqrt{\varepsilon}, & \text{if~} H=AA^T.
\end{cases}
\end{equation}

(3) When $H$ is singular, in $k=O(\kappa^+ \lambda/ t\varepsilon)$ iterations, if neither (\ref{onecon}) nor (\ref{twocon}) hold, then:

\begin{equation} \label{thmeq4ggg}
\Vert A^TA \widehat x_{k} - A^Tb \Vert \leq \begin{cases}
\sqrt{\varepsilon} \Vert A \Vert^{1/2} , & \text{if ~} H=A\\
\sqrt{\varepsilon}, & \text{if~} H=AA^T,
\end{cases}
\end{equation}
where $\widehat x_k= \widehat x''$ as defined in (\ref{widehatxB}).\\

(4) Suppose $Ax=b$ is solvable and $H=AA^T$ and $x_0=A^Tw_0$ for some $w_0 \in \mathbb{R}^m$.  If ${\rm rank}(A)=m$, $x_k$ converges to the minimum-norm solution.
For arbitrary rank $A$, for each $k \geq 1$,  $x_k=A^Tw_k$ for some $w_k \in \mathbb{R}^m$ and $\widehat x_k=A^T \widehat w_k$ for some $\widehat w_k \in \mathbb{R}^m$.  Moreover, if $\{w_k\}$ is a bounded sequence, then $x_k$ converges to the minimum-norm solution of $Ax=b$. Likewise, if $\{\widehat w_k\}$ is a bounded sequence, then $\widehat x_k$ converges to the minimum-norm solution of $Ax=b$.\\

(5) Each iteration takes $O(Nt+t^3)$ operations, $N$ the number of nonzeros entries in $A$.\\

\end{theorem}

\begin{remark}  According to Theorem \ref{thmone}  when $H$ is invertible  the worst-case bound dependence of CTA on $\varepsilon$ is proportional to $\ln (1/\varepsilon)$. When $H$ is singular the worst-case bound is proportional to $1/\varepsilon$ or $1/\varepsilon^2$ when $Ax=b$ is solvable. Nevertheless, on the one hand a powerful aspect of the iterations of $F_t$ is that whether or not $A$ is invertible, full-rank  or arbitrary rank, bounds are available to compute approximate solution of the normal equation.   On the other hand, the worst-case bounds are not tight so that the actual average performance must be much better.  If $A$ is not invertible or full row-rank but $Ax=b$ is solvable, there is no need to compute $\widehat x_k$ so that Algorithm \ref{1.1} need not be modified. However, if $Ax=b$ is not solvable and in each iteration of the algorithm we compute $\widehat x_k$, then the iteration bound is faster by a factor of $1/t$ so that it may well be worth computing it.  In fact as discussed in \cite{kalantari23},  the derived complexity bounds
are not optimal and our computational results support this. From the practical point of view, Algorithm \ref{1.1} has many desired features.  Even for small values of $t$ it perform quite well on variety of matrices tested. Finally, we remark that while part (4) of the theorem states that under some conditions $x_k$ converges to the minimum-norm solution of $Ax=b$, CTA does not offer such a result for computing the mininmum-norm solution to the normal equation. However, a version of the Triangle Algorithm from \cite{kalantari23} to be described in the next section has the property that given an $\varepsilon$-approximate solution it computes an $\varepsilon$-approximate minimum-norm solution as defined in Definition \ref{minnormapp}
\end{remark}

The next theorem offers properties of collective use of the CTA family proved in \cite{kalantari23}.\\

\begin{theorem}  Given $r_0=b-Ax_0 \not =0$, suppose its  minimal polynomial with respect to $H$ has degree $s$. Then the {\it point-wise orbit}, $\{F_1(r_0), \dots, F_t(r_0) \}$ satisfies the following:

(i) $\Vert F_1(r_0) \Vert > \cdots >  \Vert F_{s}(r_0) \Vert=\Vert F_{s+1}(r_0) \Vert = \cdots = \Vert F_{m}(r_0) \Vert$.

(ii) $Ax=b$ is solvable if any only if $F_{s}(r_0)=0$.

(iii) $A^TAx=A^Tb$ is solvable exclusively if and only if $F_{s}(r_0) \not= 0$ and $A^T F_s(r_0)=0$.

(iv) The sequence is computable in $O(Ns+s^3)=O(m^2(m+n))$ operations.
\end{theorem}

\subsection{The Dynamics of Iterations of First-Order CTA} \label{sec3xx}
Before presenting the computational results, we study the dynamics of the iterates of $F_1$.  We consider solving $Ax=b$ via iterations of $F_1(r)=r- (r^T H r)/r^T H^2 r) H r$, where as before  $H=AA^T$ when $A$ is a general $m \times n$ matrix, and $H=A$ when $A$ is $m \times m$ symmetric PSD.  Since we wish to analyze some properties of the dynamics of iterations of $F_1(r)$ we ignore the fact that each residual is assumed to be of the form $r_k=b-Ax_k$ since $x_k$ does not play a role in convergence.

\begin{definition}

\noindent (i) To each $r_0 \in \mathbb{R}^m$ we associate a sequence  $\{r_k=F_1(r_{k-1}): k=1, 2, \dots\}$, referred as the {\it residual orbit} at $r_0$ and denote it by $O^+(r_0, H)$.

\noindent (ii) Suppose  $H=U \Lambda U^T$, where $U=[u_1, \dots, u_m]$ is the matrix of orthogonal eigenvectors and $\Lambda = {\rm diag}(\lambda_1, \dots, \lambda_m)$ the diagonal matrix of eigenvalues. Assume $H$ is positive definite and without loss of generality $0 < \lambda_1 \leq \lambda_2 \leq \cdots \leq \lambda_m$.   Given $r_0 \in \mathbb{R}^m$, let $y_0=U^Tr_0$ and let $O^+(y_0, \Lambda)$, i.e.  the orbit corresponding to the iteration of $F_1$, where given $y_k \in \mathbb{R}^m$, the next iterate is defined according to  the recursive formula
\begin{equation} \label{subseceq3.3}
y_{k+1}= y_k - \frac{y_k^T \Lambda y_k}{y_k^T \Lambda^2 y_k} \Lambda y_k.
\end{equation}

\noindent (iii)  Given distinct pair of eigenvalues  $\lambda_i < \lambda_j$ with corresponding eigenvectors $u_i, u_j$, set
\begin{equation} \label{c-zero}
\alpha_i= \sqrt{\frac{\lambda_j}{\lambda_i+ \lambda_j}}, \quad
\alpha_j = \sqrt{\frac{\lambda_i}{\lambda_i+ \lambda_j}}.
\end{equation}
Let
\begin{equation} \label{c-one}
r^{+}= \alpha_i u_i + \alpha_j u_j, \quad  r^{-}= \alpha_i u_i - \alpha_j u_j.
\end{equation}
Define the {\it $ij$-critical lines} to be the pair of lines
$$l_{ij}^+=\{ \alpha r^{+} : \alpha \in \mathbb{R}\}, \quad
l_{ij}^-=\{ \alpha  r^{-} :  \alpha \in \mathbb{R}\}.$$
\end{definition}

The  theorem below sheds lights on the dynamics of the orbits under iterations of $F_1$. The convergence of orbits for points on and near the $ij$-critical lines will be slower than other locations, while it will be fastest at eigenvectors and points near them. Clearly, in dimension $2$ there is only one pair of critical lines. The convergence properties for this dimension  can actually be seen in Figure \ref{tfig} showing the dynamics for solving a $2 \times 2$ positive definite diagonal matrix $A$.

\begin{figure}[!htb]
\centering
\includegraphics[width=.32\linewidth]{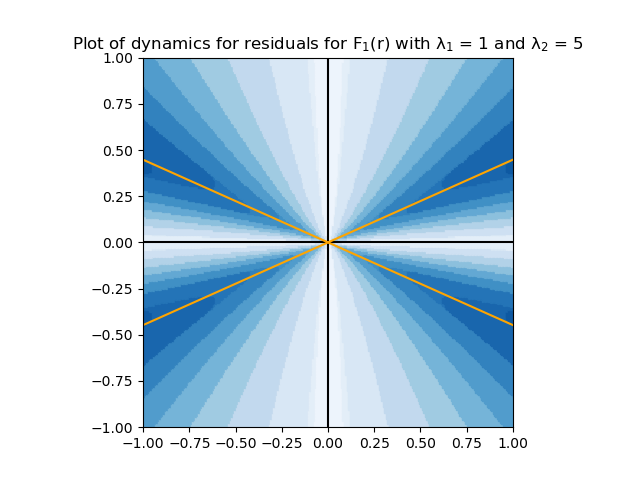} \includegraphics[width=.32\linewidth]{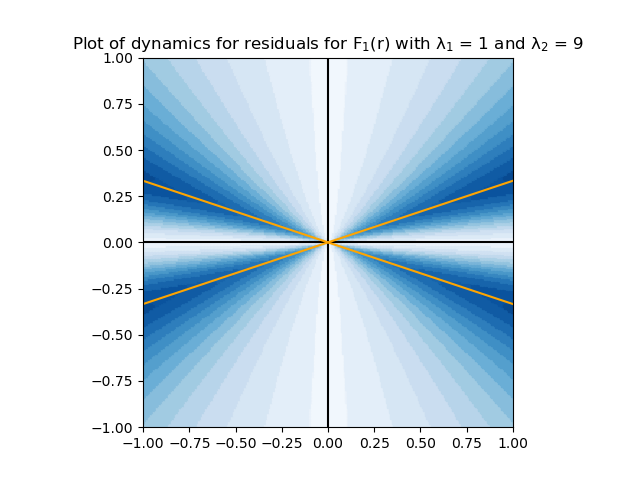}
\includegraphics[width=.32\linewidth]{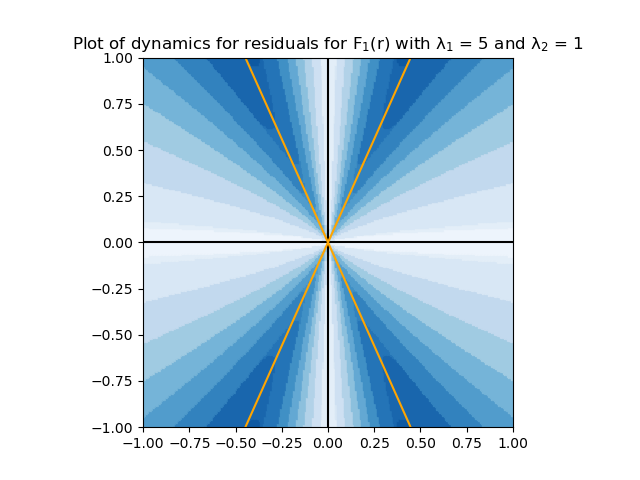}\\ \includegraphics[width=.32\linewidth]{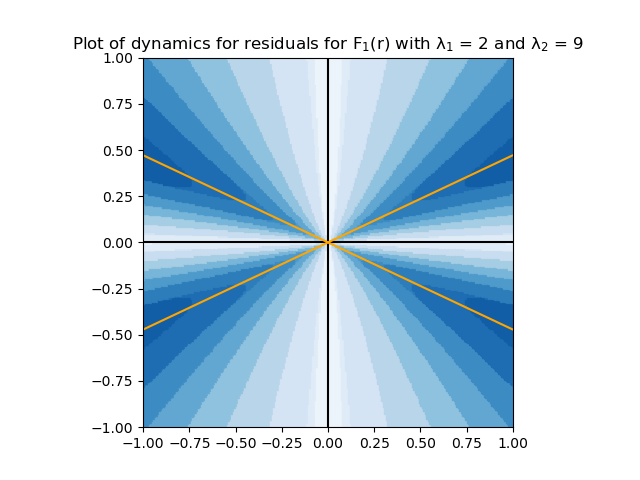}
\includegraphics[width=.32\linewidth]{Figure_1_5.png} \includegraphics[width=.32\linewidth]{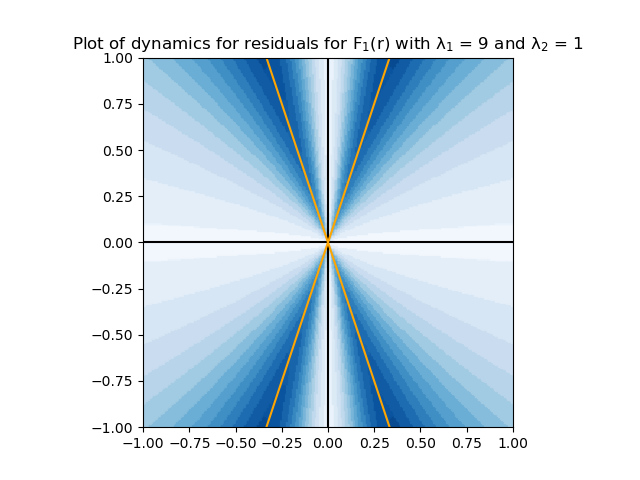}\\
\caption{Dynamics of  iterations of $F_1$ with respect to a $2 \times 2$ diagonal PSD matrix. The eigenvalues are given on top of each image. The critical lines are drawn as yellow crossing lines. The axis correspond to eigenvectors.}
\label{tfig}
\end{figure}

\begin{theorem} \label{thmdynamic} Assume $H$ is positive definite with condition number $\kappa$.\

(i) $F_1(r)$ is {\it homogeneous} of degree one: for any scalar $\alpha$ and $r \in \mathbb{R}^m$, $F_1(\alpha r) = \alpha F_1(r)$. In particular, for any nonzero $r_0 \in \mathbb{R}^m$,  the orbit at $r_0$ is $\Vert r_0 \Vert $ times the orbit at $r_0/\Vert r_0 \Vert$.  It follows that the dynamics of $F_1(r)$ on the unit ball is identical with its dynamic on a ball of arbitrary positive radius at the origin.

(ii) Given $r_0$ and $y_0=U^Tr_0$, for all $k$, $\Vert r_k \Vert = \Vert y_k \Vert$  and if $r_k \not =0$,
\begin{equation} \label{subseceq3.4}
\frac{\Vert r_{k+1} \Vert^2}{\Vert r_k \Vert^2 }=1- \frac{(r_k^T H r_k)^2}{r_k^T H^2 r_k} \frac{1}{ \Vert r_k \Vert^2}= \frac{\Vert y_{k+1} \Vert^2 }{\Vert y_k \Vert^2}=1- \frac{(y_k^T\Lambda y_k)^2}{y_k^T \Lambda^2 y_k} \frac{1}{ \Vert y_k \Vert^2}.
\end{equation}
In particular, $O^+(r_0, H)$ and  $O^+(y_0, \Lambda)$ both converge to zero having the same convergence rate.

(iii) Suppose $r_0 \not =0$ is an eigenvector of $H$, say $Hr_0= \lambda r_0$. Then $r_1=F_1(r_0)=0$ so that it converges in one step. Suppose $r_0$ is an $\varepsilon$-approximate unit eigenvector, i.e.
\begin{equation} \label{eigapp}
H r_0 =  \lambda r_0+ e, \quad \Vert r_0 \Vert =1, \quad  \Vert e \Vert \leq \varepsilon.
\end{equation}
If $\varepsilon \leq \lambda^2/ 4 \kappa$, then
\begin{equation} \label{subseceq3.6}
\Vert r_1 \Vert \leq  2\sqrt{\frac{\kappa\varepsilon}{\lambda}}.
\end{equation}

(iv) For any $m \geq 2$ there are at most $m(m-1)$ distinct critical lines.
Suppose $r_k$ is on a $ij$-critical line, $\lambda_i < \lambda_j$. Set
$$\kappa_{ij} = \frac{\lambda_j}{\lambda_i}, \quad
\rho_{ij}= \frac{\lambda_j - \lambda_i}{\lambda_j +\lambda_i}= \frac{\kappa_{ij}-1}{\kappa_{ij}+1}.$$
Then
\begin{equation} \label{c-two}
F_1(\alpha r^{+})=  \alpha \rho_{ij} r^{-}, \quad F_1(\alpha r^{-})=  \alpha \rho_{ij} r^{+}.
\end{equation}
Thus the orbit of points on the critical lines zig-zag and if $r_0$ lies on the $ij$-critical lines,  the $k$-th residual satisfies
\begin{equation} \label{c-three}
{\Vert r_{k} \Vert} = \rho_{ij}^k \Vert r_{0} \Vert.
\end{equation}
In particular, if $r_0$ is on the $1n$-critical line (i.e. $i=1$, $j=n$) the worst-case complexity of $F_1$ is attained:
\begin{equation} \label{c-five}
\Vert r_{k} \Vert= \rho_{1n}^k \Vert r_{0} \Vert= \bigg (\frac{\kappa-1}{\kappa+1} \bigg )^k  \Vert r_{0} \Vert.
\end{equation}
However, if $r_k$ is on an $ij$-critical line the convex combination $\overline r_k =\alpha_* r_k + (1-\alpha_*) r_{k+1}$,  $\alpha_* = \rho_{ij}/(\rho_{ij}+1)$, satisfies $F_1(\overline r_k)=0$.
\end{theorem}

\begin{proof}\

(i): This is straightforward.

(ii):  Using induction on $k$, and substituting $r_k=Uy_k$ it follows that  $r_{k+1}=Uy_{k+1}$. Hence $y_{k+1}=U^Tr_{k+1}$.  Next using that $\Vert r_k \Vert = \Vert y_k \Vert$ and substituting  $\Vert r_{k+1} \Vert^2= r_{k+1}^Tr_{k+1}$, $\Vert y_{k+1} \Vert^2= y_{k+1}^Ty_{k+1}$ we get (\ref{subseceq3.4}).

(iii):  If $r_0$ is an eigenvector, trivially $F_1(r_0)=0$. If $r_0$ is an approximate eigenvector satisfying (\ref{eigapp}),
\begin{equation}  \label{eqxx1}
H^2r_0= \lambda^2 r_0+ \lambda e + He, \quad
r_0^THr_0=\lambda \Vert r_0 \Vert^2+r_0^Te= \lambda + r_0^Te, \quad
r_0^TH^2r_0=\lambda^2+ \lambda r_0^Te+r_0^THe.
\end{equation}
Using (\ref{eqxx1}) and since $\Vert r_0 \Vert =1$,  we get
\begin{equation}  \label{eqxx4}
\Vert r_1 \Vert^2= 1- \frac{(r_0^THr_0)^2}{r_0^TH^2r_0} = 1- \frac{\lambda^2 + 2 \lambda r_0^Te + (r_0^Te)^2}{\lambda^2 + \lambda r_0^Te +r_0^T He}.
\end{equation}
From the inequality in (\ref{eigapp}) and Cauchy-Schwarz inequality we get
\begin{equation}  \label{eqxx5}
|r_0^Te| \leq \Vert r_0 \Vert \cdot \Vert e \Vert \leq \varepsilon, \quad
r_0^THe \leq \Vert r_0 \Vert \cdot \Vert H\Vert \cdot \Vert e \Vert \leq \lambda_{\max} \varepsilon.
\end{equation}
Simplifying (\ref{eqxx4}), while using (\ref{eqxx5}) we get
\begin{equation}  \label{eqx6}
\Vert r_1 \Vert^2= \frac{-\lambda r_0^Te + r_0^THe- (r_0^Te)^2}{\lambda^2 + \lambda r_0^Te +r_0^T He} \leq
\frac{\lambda \varepsilon + \lambda_{\max} \varepsilon}
{\lambda^2  - \lambda \varepsilon - \lambda_{\max} \varepsilon} \leq
\frac{2\lambda_{\max} \varepsilon}
{\lambda^2  - 2\lambda_{\max} \varepsilon} \leq \frac{4\lambda_{\max} \varepsilon}{\lambda^2} \leq \frac{4\kappa \varepsilon}{\lambda}.
\end{equation}
Hence the proof of (\ref{subseceq3.6}).

(iv):  Suppose the eigenvalues of $H$ are distinct. Then for each distinct pair $i, j$, $i < j$, there is a pair of $ij$-critical lines. Hence there are $m(m-1)$ critical lines. Next we prove (\ref{c-two}).
Since $F_1(r)$ is homogeneous, without loss of generality we assume $\alpha =1$.  From (\ref{c-one}) if $r_k=r^+$, we get
$H r_k = \alpha_i \lambda_iu_i + \alpha_j \lambda_j u_j$,
$H^2 r_k = \alpha_i \lambda^2_iu_i + \alpha_j \lambda^2_j u_j$,
$r_k^THr_k= \frac{2 \lambda_i \lambda_j}{\lambda_i+ \lambda_j}$ and  $r_k^TH^2r_k= \lambda_i \lambda_j$.
From these we get
\begin{equation} \label{c-eight}
F_1(r_k)= \alpha_i u_i + \alpha_j u_j - \frac{2}{\lambda_i + \lambda_j}  (\alpha_i \lambda_iu_i + \alpha_j \lambda_j u_j ).
\end{equation}
Note that
\begin{equation} \label{c-nine}
\alpha_i - \frac{2}{\lambda_i + \lambda_j} \alpha_i\lambda_i
=\rho_{ij} \alpha_i, \quad
\alpha_j - \frac{2}{\lambda_i + \lambda_j} \alpha_j \lambda_j
= -\rho_{ij}\alpha_j.
\end{equation}
It follows that $F_1(r^+)= \rho_{ij}  ( \alpha_i u_i -  \alpha_j u_j)=\rho_{ij} r^{-}$. Similarly  $F_1(r^{-})=\rho_{ij} r^{+}$.  Hence the proof of (\ref{c-two}).  From this the proof of (\ref{c-three}) and (\ref{c-five}) follow.  To prove last part of (vi),  substitution for $\alpha_*$,
if $r_k=\alpha r^{+}$, $\overline r_k =\alpha_* r_k + (1-\alpha_*) r_{k+1}=  \alpha u_i$ and if  $r_k=\alpha r^{-}$,
$\overline r_k =\alpha_* r_k + (1-\alpha_*) r_{k+1}=  \alpha u_j$. In either case $F_1(\overline r_k)=0$.
\end{proof}

\begin{remark} As proved in \cite{kalantari23}, the worst-case of $F_1$ occurs when the initial point is on the $1n$-critical line. The practical application of the above theorem on the dynamics of iterations of $F_1$ is that if computationally we witness that the magnitudes of $r_k$ and $r_{k+1}$ do not differ by much we can take as the next iterate a convex combination of $\alpha r_k +(1+ \alpha) r_{k+1}$ and compute $\alpha$ as
$$\alpha_* = {\rm argmin} \{ \Vert \alpha  F_1(r_k) + (1- \alpha) F_1(r_{k+1}) \Vert, \quad \alpha \in [0,1] \},$$
simply by minimization of a quadratic equation in the interval $[0,1]$. In fact this can be done with iterations of any $F_t$, $t \leq m$.  
More generally one would expect some aspects of the above theorem can be extended for general $F_t(r)$. However, as discussed in \cite{kalantari23} the exact worst-case analysis of $F_t$ for $t >1$ appears to be a complicated optimization problem - open to further investigation - as is the characterization of critical regions as lines or curves.
\end{remark}

\section{The Triangle Algorithm (TA)} \label{sec4}
In this section we describe the {\it Triangle Algorithm} (TA) for solving the approximation problems defined earlier from \cite{kalantari23}. The TA to be described here is specialized version of an algorithm
for the {\it convex hull membership problem} (CHMP):
Given a compact subset $S$ in $\mathbb{R}^m$, a distinguished point $p \in \mathbb{R}^m$, and  $\varepsilon \in (0, 1)$, either compute  
$p' \in C=conv(S)$, the {\it convex hull} of $S$, such that $\|p' - p\| \leq \varepsilon$, or find a hyperplane that separates $p$ from $C$.

\subsection{Technical Foundation of TA}
\label{sec2.1}
In the case of applying the Triangle Algorithm to solving a linear system $Ax=b$, the compact convex set $C$ is the ellipsoid below, the image of the central ball of radius $\rho$ under the linear mapping defined by $A$:
\begin{equation}
E_{A, \rho}= \{ Ax: \Vert x \Vert \leq \rho \}.
\end{equation}
Also, the point $p$ is $b$.  In what follows we state several properties of the general TA in the context of testing if $b \in E_{A, \rho}$. 
Then we describe a version of the general Triangle Algorithm specialized to tests if $b \in E_{A, \rho}$ to within tolerance of $\varepsilon$. 

\begin{definition} \label{defpTA} Given $\rho >0$, $x' \in \mathbb{R}^n$ such that $b'=Ax' \in E_{A, \rho}$, $b' \not =b$,  called {\it iterate}, a point $v \in E_{A, \rho}$ is called a {\it pivot} (at $b'$) if
\begin{equation} \label{pivotTA}
\|b' - v\| \geq \|b - v\| \quad \iff \quad
(b-b')^Tv \geq \frac{1}{2} \big (\Vert b \Vert^2 - \Vert b' \Vert^2 \big).
\end{equation}
If no pivot exists, $b'$ is called a $b$-{\it witness} (or just {\it witness}). A pivot $v$ is called a {\it strict pivot} if
\begin{equation} \label{pivotsTA}
(b-b')^T(v-b) \geq 0.
\end{equation}
In particular, when the three points are distinct, $\angle b'bv$ is at least $\pi/2$.
\end{definition}

The Triangle Algorithm (TA) works as follows: Given $\varepsilon \in (0,1)$, $\rho >0$ and  $x' \in \mathbb{R}^n$ such that $b' \in E_{A, \rho}$, if $\Vert b - b' \Vert \leq \varepsilon$ it stops. If $b'$ is a witness, then $b \not \in E_{A, \rho}$. Otherwise, it computes a pivot $v$ and the next iterate $b''$ as the nearest point to $b$ on the line segment $b'v$. Specifically,

\begin{prop} \label{prop1TA} If $v$ is a strict pivot,
\begin{equation} \label{alph}
    b'' = (1 - \alpha)b' + \alpha v, \quad  \alpha = \frac{(b - b')^T(v - b')}{\|v - b'\|^2}. \qed
\end{equation}
\end{prop}

Triangle Algorithm replaces $b'$ with $b''$ which will necessarily lie in $E_{A, \rho}$. It replace $x'$ with a corresponding $x''$ so that $b''=Ax''$ and repeats the above iteration. The correctness of the TA is due to the following theorem. 

\begin{theorem} \label{thm1TA} {\rm {(Distance Duality)}}
$b \in E_{A, \rho}$ if and only if for each $b' \in E_{A, \rho} \setminus \{b\}$ there exists a (strict) pivot $v \in E_{A, \rho}$. Equivalently, $b \not\in E_{A, \rho}$ if and only if there exists a witness $b' \in E_{A, \rho}$. \qed
\end{theorem}

From the Distance Duality theorem we get the following useful lower bound on the norm of the minimum-norm solution, $x_*$, based on a computation of a $b$-witness. Later we will see it's algorithmic utility.

\begin{cor} \label{cor1} Consider the linear system $Ax=b$. Let $x_*$  be its minimum-norm least-squares solution. If for a given $\rho >0$, $b' \in E_{A,\rho}$ is a $b$-witness, then
\begin{equation}
\frac{(b-b')^Tb}{ \Vert A^T(b-b') \Vert} < \Vert x_* \Vert.
\end{equation}
\end{cor}

To complete the description of TA, we need to compute a pivot
(or strict pivot). This is achieved by computing the optimal solution of the following optimization problem:
\begin{equation} \label{pivottestTA}
v_\rho= {\rm argmax}\{c^Tx : Ax \in E_{A, \rho}\}, \quad c = A^T(b - b').
\end{equation}

Geometrically, to find $v_\rho$, consider the orthogonal hyperplane to line $bb'$, then move it from $b'$ toward $b$ until it is tangential to the boundary of $E_{A, \rho}$. For illustration, see Figure \ref{AAD}. It follows from (\ref{pivotTA}) that either $v_\rho$ is a pivot, or $b'$ is a witness. In fact, if $b \in E_{A, \rho}$,  $v_\rho$ is a strict pivot.

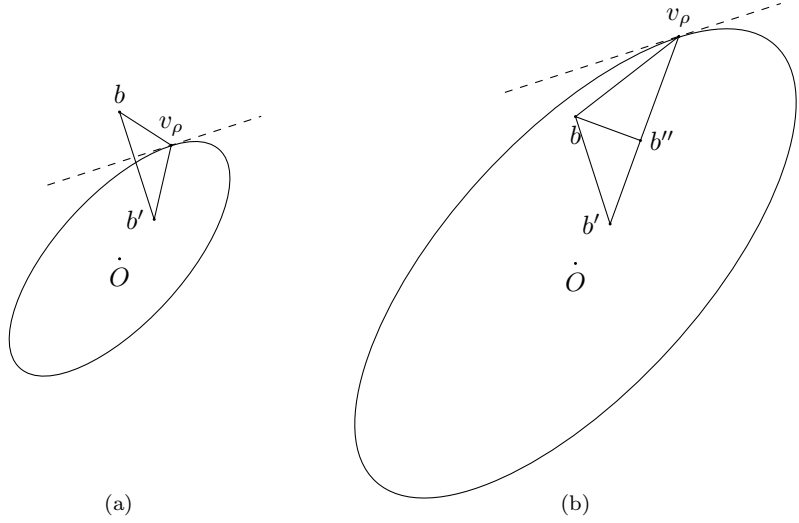
\begin{figure}[htpb]
	\centering
	\subfloat[] {
	\raisebox{10ex}{
	\begin{tikzpicture}[scale=.7]
	\begin{scope}
	\begin{scope}
	[rotate=47.5829]
	\draw \boundellipse{(0,0}{2.810}{1.210};
	\filldraw (0,0) circle (.5pt) node[below] {$O$};
	\end{scope}
	\filldraw (0,2.79) circle (.5pt)  node[above] {$b$};
	\filldraw (0.9815,2.158)  circle (.5pt)  node[above] {$v_\rho$};
	\filldraw (0.6580,0.75)  circle (.5pt)  node[left] {$b'$};
	\draw[dashed] (-1.37,1.4)--(2.69,2.709);
	\draw (0.6580,0.75)--(0.9815,2.158)--(0,2.79)--(0.6580,0.75);
	\end{scope}
	\label{1st}
	\end{tikzpicture}}}
	\qquad
	\subfloat[] {
	\begin{tikzpicture}[scale=.7]
	\begin{scope}
	\begin{scope}
	[rotate=47.5829]
	\draw  \boundellipse{(0,0}{5.620}{2.420};
	\filldraw (0,0) circle (.5pt) node[below] {$O$};
	\end{scope}
	\filldraw (0,2.79) circle (.5pt)  node[below] {$b$};
	\filldraw (1.963,4.315)  circle (.5pt)  node[above] {$v_\rho$};
	\filldraw (1.234,2.336)  circle (.5pt)  node[right] {$b''$};
	\filldraw (0.6580,0.75)  circle (.5pt)  node[left] {$b'$};
	\draw[dashed] (-1.34,3.25)--(3.9,4.94);
	\draw (1.963,4.315)--(1.234,2.336)--(0.6580,0.75)--(0,2.79)--(1.963,4.315);
	\draw (0,2.79)--(1.234,2.336);
	\end{scope}
	
	\end{tikzpicture}
	\label{2nd}
	}
	
	\caption{In Figure \ref{1st}, $v_\rho$ is not a strict pivot, proving $b$ is exterior to the ellipsoid.
In Figure \ref{2nd},  $b'$ admits a pivot, $v_r$, used to compute the next iterate $b''$  as the nearest point to $b$ on the line segment $b'v_\rho$. Geometrically, $v_\rho$ is found by moving the orthogonal hyperplane to line $bb'$ (dashed line) in the direction from $b'$ to $b$ until the hyperplane is tangential to the ellipsoid.}
	\label{AAD}
\end{figure}

Testing if $b \in E_{A, \rho}$  is an important subproblem for solving $Ax=b$ or the normal equation.  In what follows we first state Theorem \ref{prop0TA} regarding the solvability of $Ax=b$ in the ellipsoid $E_{A, \rho}$.
Then in Theorem \ref{newTA} we characterize the unsolvability of $Ax=b$ and its  Corollary \ref{cor1} results in a lower bound on the minimum-norm solutions $x_*$ when a witness is at hand.  We will then state Algorithm \ref{2.1} and state the complexity of the Triangle Algorithm for testing if $b \in E_{A, \rho}$ to within $\varepsilon$.

\begin{theorem} \label{prop0TA} Given $x' \in \mathbb{R}^n$, where $\|x'\| \leq \rho$, let $b' = Ax'$. Assume $b' \not =b$. Let
\begin{equation} \label{c}
c= A^T(b-b')=A^TAx'-A^Tb.
\end{equation}

(1) If $c=0$, $x'$ is a solution to normal equation, $Ax=b$ is unsolvable and $b'$ is a witness, proving $b \not \in E_{A, \rho}$.

(2) Suppose $c \not =0$. Then
\begin{equation}  \label{vr}
v_\rho= \rho \frac{Ac}{\Vert c \Vert}= {\rm argmax}\{c^Tx : Ax \in E_{A, \rho}\}, \quad i.e. \quad
\max\{c^Tx: Ax\in  E_{A, \rho} \} = c^T v_\rho= \rho \Vert c \Vert.
\end{equation}

(3) $v_\rho$ is a strict pivot if and only if
\begin{equation}  \label{eqlem2}
\rho \Vert c \Vert \geq (b-b')^Tb.
\end{equation}
\end{theorem}

\begin{remark}  Algorithmically, an important aspect of the Triangle Algorithm relies on the fact that the optimization of a linear function over an ellipsoid can be computed efficiently, see (\ref{vr}).
\end{remark}

\begin{theorem}  \label{newTA} $Ax=b$ is unsolvable if and only if for each $\rho >0$ there exists a witness $b' \in E_{A,\rho}$ such that
\begin{equation} \label{infeasibilityTA}
\rho \Vert c \Vert < (b-b')^Tb, \quad c=A^T(b-b').
\end{equation}
\end{theorem}

\subsection{TA for Testing Solvability of a Linear System over Fixed Ellipsoid}
We may now give the description of the first version of the  Triangle Algorithm, Algorithm \ref{2.1}.

\begin{algorithm}[!htb]
\SetAlgoNoLine
\KwIn{$A \in \mathbb{R}^{m \times n}$, $b \in \mathbb{R}^m$, $b \not =0$, $\rho>0$, $\varepsilon, \varepsilon' \in (0, 1)$.}
$x' \gets 0$, $b' \gets 0$

\While{$(\|b - b'\| > \varepsilon)$ $\wedge$ $(\Vert A^T(b-b') \Vert > \varepsilon')$}{$c=A^T(b-b')$,
$v_\rho=\rho
A\frac{c}{\|c\|}$.

  \lIf{$ \rho \Vert c \Vert \geq (b-b')^Tb$} { $\alpha \gets (b - b')^T(v_\rho - b')/\|v_\rho - b'\|^2$,

  $b' \gets (1-\alpha) b' + \alpha v_\rho$, \quad $x' \gets (1-\alpha) x' +  \alpha \frac{\rho c}{\Vert c \Vert}$}
   \Else{{$b'=Ax'$ is a witness, STOP}}}
    \caption{(TA) Computes an $\varepsilon$-approximate solution of $Ax = b$ or an $\varepsilon'$-approximate solution of $A^TAx = A^Tb$, or a witness $b'$ in $E_{A, \rho}$ (if no witness, final $x'$ is desired approximate solution) .} \label{2.1}
\end{algorithm}

From Theorem \ref{prop0TA}, Corollary \ref{cor1} and the structure of Algorithm \ref{2.1} we have:

\begin{prop} \label{propxTA} Algorithm \ref{2.1} either terminate with
$x'$ such that
$\Vert Ax' - b \Vert \leq \varepsilon$, or $ \Vert A^TAx'-A^Tb \Vert \leq \varepsilon'$, or $b'=Ax'$ a witness, proving that $b \not \in E_{A, \rho}$.
Furthermore, if $b'$ is a witness then $\frac{(b-b')^Tb}{ \Vert A^T(b-b') \Vert} < \Vert x_* \Vert$. \qed
\end{prop}

\begin{remark}  When $Ax=b$ admits an $\varepsilon$-approximate solution in $E_{A,b}$ the algorithm may terminate sooner with an $\varepsilon'$-approximate solution of the normal equation.  By adjusting the precision, Algorithm \ref{2.2}, to be described later, will be able to identify an $\varepsilon$-approximate solution of $Ax=b$.
\end{remark}

The iteration complexity of Algorithm \ref{2.1}  is based on  the iteration complexity of the general Triangle Algorithm for solving CHMP.

\begin{theorem} \label{thm2pTA} {\rm {(Iteration Complexity Bounds)}}
Algorithm \ref{2.1} terminates with  $b' \in E_{A,\rho}$ satisfying one of the  following conditions:

(i) $\|b - b'\| \leq \varepsilon$ and the number of iterations to compute $b'$ is $O\big (\rho^2 \Vert A \Vert^2/\varepsilon^2 \big )$.

(ii) $b'$ is witness and the number of iterations to compute it is $O(\rho^2 \Vert A \Vert^2/\delta_\rho^2)$, where $\delta_\rho = \min \{\|Ax - b\| :  \Vert x \Vert \leq \rho\}$. Moreover,
\begin{equation} \label{deltaboundTA}
\delta_\rho \leq \|b - b'\| \leq 2\delta_\rho.
\end{equation}

(iii) Suppose for some $\rho_b >0$ the intersection of ball of radius $\rho_b$ centered at $b$ and $E_{A,\rho}$ is a ball of the same radius in the interior of
$E_{A,\rho}$.
Then the number of iterations is $O\big ((\rho \Vert A \Vert/\rho_b)^2 \ln{1 / \varepsilon} \big )$.
\end{theorem}

\begin{remark} Part (ii) implies when $b \not \in E_{A, \rho}$ the complexity of finding a witness is dependent on the distance from $b$ to the ellipsoid $E_{A, \rho}$, the farther the distance, the fewer iterations it takes to detect unsolvability. Part (iii) suggests when $b$ is well situated in $E_{A, \rho}$ the number of iterations is essentially proportional to the square of the ratio $\rho/\rho_b$. This ratio is analogous to a {\it condition number} associated with the particular problem. In contrast the number of iterations depends logarithmically on $1/\varepsilon$.
\end{remark}

\subsection{TA for Solving a Linear System or Normal Equation}
Next, in Algorithm \ref{2.2} we modify Algorithm \ref{2.1} so that it either computes an $\varepsilon$-approximate solution of $Ax=b$ or an $\varepsilon'$-approximate solution of $A^TAx=A^Tb$. For this reason we need to search in an ellipsoid $E_{A, \rho}$ so that $ \rho$ is at least $\Vert x_* \Vert$, the minimum-norm solution of the normal equation.
Next, each time a given $\rho$ results in a witness $b'$, using Theorem \ref{prop0TA} part (3), we increase $\rho$ sufficiently so that there is a strict pivot for that witness. However, we also want to make sure that $\rho$ increases sufficiently to eventually catch up to $\Vert x_* \Vert$. Thus we take the new $\rho$ to be the maximum of $2 \rho$ and $(b-b')^Tb/\Vert c \Vert$.

\begin{algorithm}[!htb]
\SetAlgoNoLine
\KwIn{$A \in \mathbb{R}^{m \times n}$, $b \in \mathbb{R}^m$, $b \not =0$, $\varepsilon, \varepsilon' \in (0,1)$ }
$\rho \gets 0$, $x' \gets 0$, $b' \gets 0$.

\While{$(\|b - b'\| > \varepsilon)$ $\wedge$  $( \Vert A^T(b-b') \Vert > \varepsilon')$}{$c=A^T(b-b')$,
$v_\rho \gets \rho
A\frac{c}{\|c\|}$.

  \lIf{$ \rho \Vert c \Vert \geq (b-b')^Tb$} { $\alpha \gets (b - b')^T(v_\rho - b')/\|v_\rho - b'\|^2$,

   $b' \gets (1-\alpha) b' + \alpha v_\rho$, \quad $x' \gets (1-\alpha) x' +  \alpha \frac{\rho c}{\Vert c \Vert}$}
   \Else{{$\rho \gets \max \{2 \rho,  \frac{(b-b')^Tb }{\Vert c \Vert}\}$}}}
    \caption{(TA) Computes an $\varepsilon$-approximate solution of $Ax=b$ or $\varepsilon'$-approximate solution of $A^TAx=A^Tb$ (final $x'$ is  approximate solution).} \label{2.2}
\end{algorithm}

\begin{theorem}  \label{finalthm} Algorithm \ref{2.2} satisfies the following properties:

(i)  The value of $\rho$ in the while loop eventually will satisfy
$\rho  \geq \Vert x_* \Vert$. Then  Algorithm \ref{2.2} either computes an $\varepsilon$-approximate solution of $Ax=b$ satisfying $\Vert x_\varepsilon \Vert \leq \rho$, or an $\varepsilon'$-approximate solution  of $A^TAx=A^Tb$ satisfying $\Vert x_{\varepsilon'} \Vert \leq \rho$.

(ii)  Let  $\rho_{\varepsilon'} =  {\Vert b \Vert^2}/{\varepsilon'}$.  Suppose the value of $\rho$ in the while loop satisfies $\rho \geq 2\rho_{\varepsilon'}$. Then, if Algorithm \ref{2.2} does not compute an $\varepsilon$-approximate solution $x_\varepsilon$ of $Ax=b$  satisfying $\Vert x_\varepsilon  \Vert \leq \rho$, it computes an $\varepsilon'$-approximate solution $x_{\varepsilon'}$ of $A^TAx=A^Tb$  satisfying $\Vert x_{\varepsilon'} \Vert \leq \rho$.
\end{theorem}

\begin{remark}  Initially we may choose $\varepsilon'=\varepsilon$. If $Ax=b$ admits an $\varepsilon$-approximate solution,  Algorithm \ref{2.2} may still terminate with an $\varepsilon'$-approximate solution to the normal equation. In such case we can halve the value of $\varepsilon'$ and restart the algorithm with the current approximate solution. There are many heuristic approaches to improve the speed of the algorithm.
If $\Vert Ax -b \Vert$ does not reduce much, becoming evident that $Ax=b$ is not solvable, we can input $A^TAx=A^Tb$ as the main equation, starting with the current iterate $x'$.
\end{remark}

\subsection{TA for Computing Approximation to  Minimum-Norm Solution} \label{sec2.3}

Suppose we are given $x_\varepsilon$, an $\varepsilon$-approximate
solution of $Ax=b$.  We will start with this to compute an $\varepsilon$-approximate
minimum-norm solution of $Ax=b$ (see Definition \ref{minnormapp}). That is, we wish to compute an $\varepsilon$-approximate solution, $x^*_\varepsilon$ such that  $\Vert Ax_\varepsilon^* - b \Vert \leq \varepsilon$, and
a number $\underline \rho$ such that $\Vert Ax - b \Vert \leq \varepsilon$ is not possible for any $x$ with $Ax \in E^\circ_{A, \underline \rho}$, also $\underline \rho$ must satisfy $\overline \rho - \underline \rho \leq \varepsilon$, where
$\overline \rho = \Vert x^*_\varepsilon \Vert$.  Algorithm \ref{2.5}  accomplishes this by systematically  decreasing $\overline \rho$ or increasing $\underline \rho$. The initial $\overline \rho = \Vert x_\varepsilon \Vert$, $x_\varepsilon$ the initial approximate solution, and the initial $\underline \rho =0$.   Geometrically, given an $\varepsilon$-approximate solution to $Ax=b$,
the algorithm start with the ellipsoid $E_{A, \Vert x_\varepsilon \Vert}$ as an upper-bounding ellipsoid. It also uses the origin as the starting lower-bounding ellipsoid so that $\underline \rho=0$.  It then tests if there is an $\varepsilon$-approximate solution in $E_{A, \rho}$, where $\rho=(\overline \rho + \underline \rho)/2$. If so, a new upper-bounding ellipsoid is found.  Otherwise, this ellipsoid becomes the new lower-bounding ellipsoid so that $\underline \rho$ is updated to be $\rho$ and the process is repeated. Each time we test if $E_{A, \rho}$ admits an $\varepsilon$-approximate solution we use the last witness, except for the first time where we start with $x'=0$.   Figure \ref{minnormx} shows the first iteration which consists in halving $\overline \rho= \Vert x_\varepsilon \Vert$. The corresponding ellipsoid (the small ellipsoid in the figure) does not contain an $\varepsilon$-approximate solution, hence it computes a witness. The strict pivot property of the Triangle Algorithm  allows expanding it to a larger ellipsoid whose interior does not contain $b$, resulting in a better bound $\underline \rho$. Hence the
gap $\overline \rho - \underline \rho$ is reduced by a factor better than $1/2$. Algorithm \ref{2.5} describes the process for a general matrix $A$.

\begin{algorithm}[!htb]
\SetAlgoNoLine
\KwIn{$A \in \mathbb{R}^{m \times n}$, $b \in \mathbb{R}^m$, $b \not =0$,
$\varepsilon$, $x_\varepsilon$ ($\varepsilon$-approximate solution of $Ax=b$)}
$\overline \rho \gets \Vert x_\varepsilon \Vert$, $\underline \rho \gets 0$,

$\rho \gets \frac{1}{2}(\overline \rho + \underline \rho)$,
$x' \gets 0$, $b' \gets Ax'$.

\While{$\overline \rho - \underline \rho > \varepsilon$}{
 \While{$((\|b - b'\| > \varepsilon) \wedge (\Vert A^T(b-b') \Vert > 0)$}
 {$c \gets A^T(b-b')$, $v_\rho \gets \rho
A\frac{c}{\|c\|}$.

  \lIf{$ \rho \Vert c \Vert \geq (b-b')^Tb$} { $\alpha \gets (b - b')^T(v_\rho - b')/\|v_\rho - b'\|^2$,\\
$b' \gets (1-\alpha) b' + \alpha v_\rho$, \quad $x' \gets (1-\alpha) x' +  \alpha \frac{\rho c}{\Vert c \Vert}$}
   \Else{{ STOP }}}
     \lIf{$\Vert b-b' \Vert \leq \varepsilon$} { $\overline \rho \gets \rho$}
   \Else{{$\underline \rho \gets \frac{(b-b')^Tb }{\Vert c \Vert}$}}}
    \caption{(TA) Given $\varepsilon$-approximate solution of $Ax=b$,  the algorithm computes an $\varepsilon$-approximate minimum-norm solution of
    $Ax=b$ or a solution to $A^TAx=A^Tb$ (final $x'$ is approximate solution).} \label{2.5}
\end{algorithm}

\begin{theorem} Algorithm \ref{2.5} either computes a solution to the normal equation (when $c=0$) or it computes an $\varepsilon$-approximate minimum-norm solution to $Ax=b$.  Moreover, if $k$ is smallest integer such that $\Vert x_\varepsilon \Vert/2^k \leq \varepsilon$, the number of iterations in the while loop is at most $k$. In other words the complexity of Algorithm \ref{2.5} is at most $O\big ( \ln (\Vert x_\varepsilon \Vert/\varepsilon) \big )$ times the complexity of Algorithm \ref{2.1}.
\qed
\end{theorem}

\begin{figure}[htpb]
	\centering
	\raisebox{10ex}{
	\begin{tikzpicture}[scale=.7]
	\begin{scope}
	\begin{scope}
	[rotate=47.5829]
	\draw \boundellipse{0,0}{2.81}{1.21};
	\filldraw (0,0) circle (.5pt) node[below] {$O$};
	\end{scope}
	\filldraw (0,2.79) circle (.5pt)  node[above] {$b$};
\filldraw (-.2,3.06) circle (.5pt)  node[above] {$\overline b$};
\filldraw (-2.5,-2.79)   node[above] {$E_{A, \widehat \rho}$};
\filldraw (-2,-2.4)   node[above] {$E_{A, \underline \rho}$};
\filldraw (-3.90,-4.)   node[above] {$E_{A, \overline \rho}$};
	\filldraw (0.9815,2.158)  circle (.5pt)  node[above] {$v_\rho$};
	\filldraw (0.6580,0.75)  circle (.5pt)  node[left] {$b'$};
	\draw[dashed] (-1.37,1.4)--(2.69,2.709);
	\draw (0.6580,0.75)--(0.9815,2.158)--(0,2.79)--(0.6580,0.75);
    \begin{scope}
	[rotate=47.5829]
	\draw \boundellipse{0,0}{5.62}{2.42};
	\filldraw (0,0) circle (.5pt) node[below] {$O$};
	\end{scope}
	\end{scope}
     \begin{scope}
	[rotate=47.5829]
	\draw \boundellipse{0,0}{3.653}{1.573};
	\filldraw (0,0) circle (.5pt) node[below] {$O$};
	\end{scope}
\begin{scope}
	[rotate=47.5829]
	\draw \boundellipse{0,0}{3.653*1.31}{1.573*1.31};
	\filldraw (0,0) circle (.5pt) node[below] {$O$};
	\end{scope}
	\end{tikzpicture}}
	\caption{{\small $E_{A, \overline \rho}$ is initial ellipsoid, in first iteration $E_{A, \underline \rho}$ gives a witness $b'$. It is improved to  larger lower bounding ellipsoid $E_{A, \widehat \rho}$. $\underline \rho$ is replaced with $\widehat \rho$. The next ellipsoid to  test has $\rho =(\overline \rho + \underline \rho)/2$. The process is repeated until $\overline \rho - \underline \rho \leq \varepsilon$.}}
	\label{minnormx}
\end{figure}
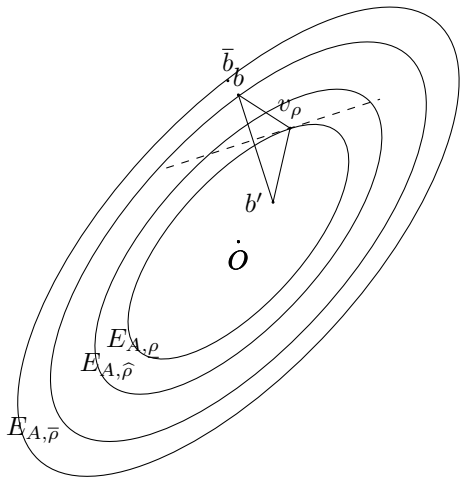

\section{Computational Results for CTA and TA} \label{sec5}
In this section, we show strong, positive, and broad computational results for CTA, in the case of square matrices, and hybrid CTA/TA algorithm, in the case of rectangular matrices. The matrices we gathered for our computational experiments encompass a variety of problems, from matrices describing problems in Computational Fluid Dynamics and ODES to linear programming (LP). Matrices are not only randomly generated, as in the case of Clement and Dorr matrices, but are also taken from well-known libraries such as Matrix Market and Suite Sparse libraries \cite{matrixmarket, suitesparse}. We first present results for square matrices, then for rectangular matrices. Experiments for both settings were run in Matlab R2022a.

In practice, popular iterative methods to solve a general (non-symmetric) linear system are the {\it Generalized minimal residual method} (GMRES), and the {\it bi-conjugate gradient method stabilized} (BiCGSTAB), a Krylov subspace method that has faster and smoother convergence than {\it conjugate gradient squared method} (CGS). For details, see Chen \cite{Chen}, Liesen and Strakos \cite{Liesen}, van der Vorst \cite{Vorst1}, Saad \cite{Saad}, Simoncini and Szyld \cite{Szyld}.

We will compare our implementations of CTA and TA to the current state-of-the-art algorithms. For the case of square matrices, we compare with GMRES, and for the case of rectangular matrices, with Matlab's implementation of ``lsqminnorm". In preliminary experimentation with symmetric positive definite matrices, we found that GMRES outperforms other methods such as BiCGSTAB and LU factorization. We therefore do not make comparisons between CTA and these other methods.  In practice GMRES is run for a number of iterations much less than the dimension and the process is repeated with the resulting approximate solution as the initial solution of the next round. This is referred as GMRES with restart. For a discussion of such approaches for solving nonsymmetric matrices for which GMRES was designed, see e.g. Baker, Jessup and Kolev \cite{Baker9} and Saad  and Schultz \cite{saad1986gmres}.  

In the current literature, there is no mention of a linear system solver implementation that is free of constraints on the coefficient matrix $A$. Our implementation of CTA and TA, to be made public in this link\footnote{{\tt https://github.com/cleunglau}}, is robust in that our implementation does not necessitate any constraints on $A$. For the case of square matrices, our implementation uses CTA. For the case of rectangular matrices, our implementation has additional subroutines for the hybrid CTA/TA algorithm. In the case of computations with large, sparse linear systems, we used Matlab's parallel computing toolbox for parallel computation. In summary, a practitioner who uses our implementation will find it easy to use. In fact, we always use the matrix $H = AA^T$. The latter is because in preliminary experimentation on symmetric positive definite matrices, we found little difference in using $H = A$ as compared to $H = AA^T$. 

\subsection{Computational Results for Square Matrices} \label{SecSqMat}
In this section, we show computational results for square matrices $A$, where $A$ is $n \times n$. Our results show that CTA outperforms GMRES, the current state-of-the-art algorithm for solving square linear systems. For an encompassing variety of matrices ranging from full-rank positive definite matrices to real world matrices used in ODEs and problems in Computational Fluid Dynamics, CTA exhibits better runtimes and quality of solutions than GMRES.

For our computations, we used our own implementation of CTA and Matlab's implementation of GMRES, solving $Ax = b$ where each component of $b$ is the row-sum vector of $A$, and therefore $x = (1, ..., 1)^T$ is a solution. Before running our experiments, the parameter $\varepsilon = 10^{-15}$ was set.  GMRES in practice runs with a restart parameter much less than the dimension of the matrix.  For a discussion of such approaches for solving nonsymmetric matrices for which GMRES was designed, see e.g. Baker, Jessup and Kolev \cite{Baker9} and Saad  and Schultz \cite{saad1986gmres}.  Typically, restarted GMRES, executes a number of iterations of GMRES, then the resulting approximate solution serves as the initial guess to start the next set of iterations.

For a fair and accurate comparison of runtimes and quality of solutions, we first run GMRES on a given linear system. However, for large matrices, Matlab's implementation of GMRES might not attain a high order of accuracy, and therefore terminates with a solution where accuracy is lower than prescribed. We therefore also terminate CTA once the quality of the solution given by CTA surpasses that of GMRES. Nevertheless, our implementation of CTA is capable of producing solutions of any prescribed accuracy. 

Our implementation of CTA uses Algorithm \ref{1.1}. Given an initial $r_0$, we start with $t = 1$ to generate $r_1$, then $t = 2$ to generate $r_2$, etc. The effect of cycling is as follows: Given residual $r_0=b-Ax_0$, we compute 
$$r_k=F_k(r_{k-1}), \quad k=1, \cdots,5.$$
Now if $A$ is invertible with condition number $\kappa$, then from a result that connects the norm of to consecutive residuals in \cite{kalantari23} we will at least have
$$ \Vert r_5 \Vert \leq \Vert r_0 \Vert \bigg (  \frac{\kappa-1}{\kappa+1} \bigg )^{1+2+3+4+5}=\Vert r_0 \Vert \bigg ( \frac{\kappa-1}{\kappa+1} \bigg ) ^{15}.$$
We also found it useful to cycle back to $t = 1$. In other words, given $r_5$, we generate the next residual as $F_4(r_5)$, etc.

Our implementation uses sparse matrix computations when the matrix dimension is large, i.e. $n > 1000$. Therefore, our implementation of CTA can also be used for solving large sparse linear systems. For the implementation of GMRES, we used the function ``gmres(A,b,5)" in Matlab, with restart parameter equal to 5. 
The maximum iterations parameter was automatically set by the implementation of GMRES itself. 

\subsubsection{Generated Positive Definite Matrices} \label{SecPD_1}
The positive definite (PD) matrices generated here are taken from diagonal matrices $A$ where diagonal values range from $1$ to $3n$, where $n$ is the dimension of the matrix. 

\begin{table}[!htb]
\centering
\begin{tabular}{|ccccc|}
\hline
\multicolumn{5}{|l|}{Runtime (in seconds) for PD Matrices}                                   \\ \hline
\multicolumn{1}{|c|}{$n$}      & \multicolumn{1}{c|}{500}  & \multicolumn{1}{c|}{1000} & \multicolumn{1}{c|}{5000}  & 10000 \\ \hline
\multicolumn{1}{|c|}{CTA}  & \multicolumn{1}{c|}{0.33} & \multicolumn{1}{c|}{0.72} & \multicolumn{1}{c|}{8.31}  & 21.63 \\ \hline
\multicolumn{1}{|c|}{GMRES} & \multicolumn{1}{c|}{0.45} & \multicolumn{1}{c|}{0.91} & \multicolumn{1}{c|}{11.28} & 38.41 \\ \hline
\end{tabular}
\caption{Comparing runtime (in seconds) with CTA vs. GMRES for positive definite matrices for various dimensions $n$.}
\label{pd_fr1}
\end{table}

\begin{table}[!htb]
\centering
\begin{tabular}{|ccccc|}
\hline
\multicolumn{5}{|l|}{Number of Iterations for PD Matrices}                                                             \\ \hline
\multicolumn{1}{|c|}{$n$}      & \multicolumn{1}{c|}{500} & \multicolumn{1}{c|}{1000} & \multicolumn{1}{c|}{5000} & 10000 \\ \hline
\multicolumn{1}{|c|}{CTA}  & \multicolumn{1}{c|}{460} & \multicolumn{1}{c|}{520}  & \multicolumn{1}{c|}{2504} & 5349  \\ \hline
\multicolumn{1}{|c|}{GMRES}   & \multicolumn{1}{c|}{429} & \multicolumn{1}{c|}{936}  & \multicolumn{1}{c|}{4629} & 7300 \\ \hline
\end{tabular}
\caption{Comparing number of iterations with CTA vs. GMRES for positive definite matrices for various dimensions $n$.}
\label{pd_fr2}
\end{table}

\begin{table}[!htb]
\centering
\begin{tabular}{|ccccc|}
\hline
\multicolumn{5}{|c|}{Quality of Solutions for PD Matrices}                                                                                                           \\ \hline
\multicolumn{1}{|c|}{$n$}      & \multicolumn{1}{c|}{500}                           & \multicolumn{1}{c|}{1000}                          & \multicolumn{1}{c|}{5000}                          & 10000                         \\ \hline
\multicolumn{1}{|c|}{CTA}   & \multicolumn{1}{c|}{\num{1.0e-15}} & \multicolumn{1}{c|}{\num{1.0e-15}} & \multicolumn{1}{c|}{\num{1.0e-15}} & \num{1.0e-15} \\ \hline
\multicolumn{1}{|c|}{GMRES}     & \multicolumn{1}{c|}{\num{1.0e-15}} & \multicolumn{1}{c|}{\num{1.01e-15}} & \multicolumn{1}{c|}{\num{1.01e-15}} & \num{1.4e-15} \\ \hline
\end{tabular}
\caption{Comparing quality of solutions with CTA vs. GMRES for positive definite matrices for various dimensions $n$.}
\label{pd_fr3}
\end{table}

As noted by Tables \ref{pd_fr1}, \ref{pd_fr2}, and \ref{pd_fr3}, the results favor CTA considerably. For small dimensions $n \leq 500$, CTA is at least 30\% faster than GMRES. For large dimensions $n = 10000$, CTA is 45\% faster than GMRES. The quality of solutions for GMRES decreases to an order of $1.4 \times 10^{-15}$ for dimension 10000, whereas CTA can be used for any arbitrary precision. Figures \ref{fig:pdresid1} and \ref{fig:pdresid2} shows the comparison in relative residuals for CTA vs. GMRES for two test matrices of dimension $100 \times 100$ and $1000 \times 1000$. Not only does the relative residuals for CTA decrease faster than GMRES, but the decrease is steady, whereas the relative residuals using GMRES drop only in the last few iterations.

\begin{figure}[!htb]
\centering
\includegraphics[width=.8\textwidth]{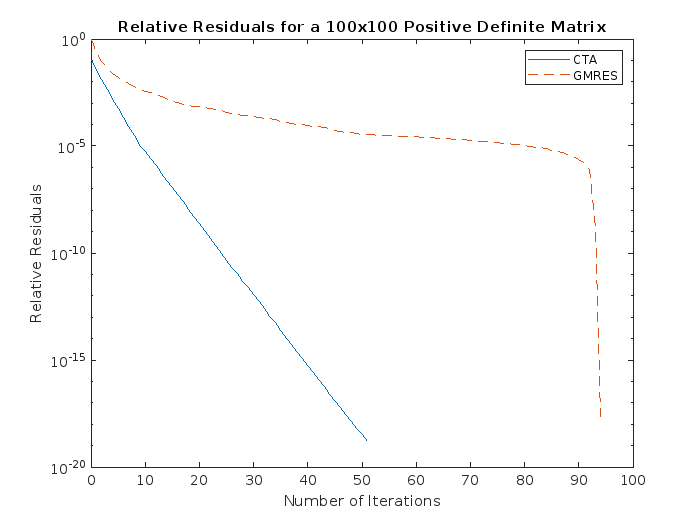}
\label{fig:pdresid1}
\caption{Comparison of relative residuals for CTA vs. GMRES for a positive definite matrix of dimension  $100 \times 100$.}
\end{figure}

\begin{figure}[!htb]
\centering
\includegraphics[width=.8\textwidth]{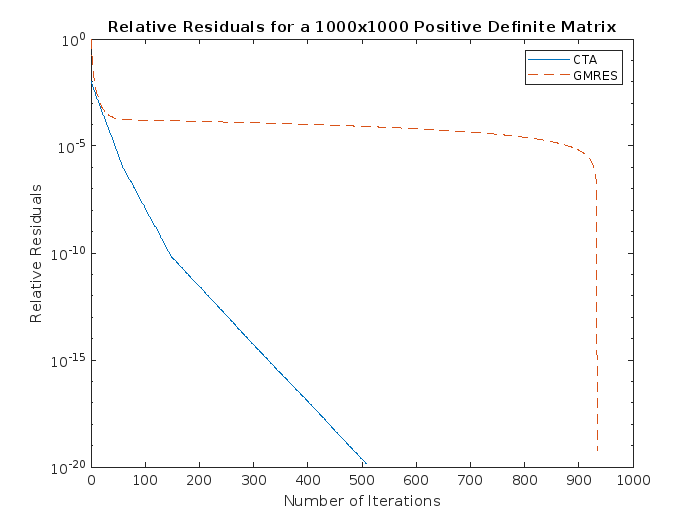}
\label{fig:pdresid2}
\caption{Comparison of relative residuals for CTA vs. GMRES for a positive definite matrix of dimension  $1000 \times 1000$.}
\end{figure}

\subsubsection{Generated Positive Semidefinite Matrices} \label{SecPSDD_1}
The positive semidefinite (PSD) matrices generated here are taken from two sets of data. The first set of data is diagonal matrices $A$ where diagonal values range from $0$ to $3n$, where $n$ is the dimension of the matrix. The second set of data is generated by the following method: Define $B$ to be a random $n \times n$ matrix generated using the Matlab command ``rand", then set $A = B^TB$. 

\begin{table}[!htb]
\centering
\begin{tabular}{|ccccc|}
\hline
\multicolumn{5}{|l|}{Runtime (in seconds) for PSD Matrices}                                   \\ \hline
\multicolumn{1}{|c|}{$n$}      & \multicolumn{1}{c|}{500}  & \multicolumn{1}{c|}{1000} & \multicolumn{1}{c|}{5000}  & 10000 \\ \hline
\multicolumn{1}{|c|}{CTA}  & \multicolumn{1}{c|}{0.34} & \multicolumn{1}{c|}{0.74} & \multicolumn{1}{c|}{9.24}  & 23.98 \\ \hline
\multicolumn{1}{|c|}{GMRES} & \multicolumn{1}{c|}{0.56} & \multicolumn{1}{c|}{1.32} & \multicolumn{1}{c|}{12.74} & 40.66 \\ \hline
\end{tabular}
\caption{Comparing runtime (in seconds) with CTA vs. GMRES for positive semidefinite matrices for various dimensions $n$.}
\label{psd_fr1}
\end{table}

\begin{table}[!htb]
\centering
\begin{tabular}{|ccccc|}
\hline
\multicolumn{5}{|l|}{Number of Iterations for PSD Matrices}                                                             \\ \hline
\multicolumn{1}{|c|}{$n$}      & \multicolumn{1}{c|}{500} & \multicolumn{1}{c|}{1000} & \multicolumn{1}{c|}{5000} & 10000 \\ \hline
\multicolumn{1}{|c|}{CTA}  & \multicolumn{1}{c|}{456} & \multicolumn{1}{c|}{501}  & \multicolumn{1}{c|}{2332} & 5101  \\ \hline
\multicolumn{1}{|c|}{GMRES}   & \multicolumn{1}{c|}{488} & \multicolumn{1}{c|}{962}  & \multicolumn{1}{c|}{4789} & 8197    \\ \hline
\end{tabular}
\caption{Comparing number of iterations with CTA vs. GMRES for positive semidefinite matrices for various dimensions $n$.}
\label{psd_fr2}
\end{table}

\begin{table}[!htb]
\centering
\begin{tabular}{|ccccc|}
\hline
\multicolumn{5}{|c|}{Quality of Solutions for PSD Matrices}                                                                                                           \\ \hline
\multicolumn{1}{|c|}{$n$}      & \multicolumn{1}{c|}{500}                           & \multicolumn{1}{c|}{1000}                          & \multicolumn{1}{c|}{5000}                          & 10000                         \\ \hline
\multicolumn{1}{|c|}{CTA}   & \multicolumn{1}{c|}{\num{1.0e-15}} & \multicolumn{1}{c|}{\num{1.0e-15}} & \multicolumn{1}{c|}{\num{1.0e-15}} & \num{1.0e-15} \\ \hline
\multicolumn{1}{|c|}{GMRES}     & \multicolumn{1}{c|}{\num{1.0e-15}} & \multicolumn{1}{c|}{\num{1.03e-15}} & \multicolumn{1}{c|}{\num{1.03e-15}} & \num{1.5e-15} \\ \hline
\end{tabular}
\caption{Comparing quality of solutions with CTA vs. GMRES for positive semidefinite matrices for various dimensions $n$.}
\label{psd_fr3}
\end{table}

As noted by Tables \ref{psd_fr1}, \ref{psd_fr2}, and \ref{psd_fr3}, the results favor CTA considerably. For small dimensions $n \leq 500$, CTA is at least 39\% faster than GMRES. For large dimensions $n = 10000$, CTA is 46\% faster than GMRES. The quality of solutions for GMRES only decreases to an order of $1.5 \times 10^{-15}$ for dimension 10000, whereas CTA can be used for any arbitrary precision. 

\subsubsection{Indefinite Matrices} \label{SecPSDD_2}
We also generated indefinite diagonal matrices: diagonal matrices $A$ where diagonal values contains positive, negative, and zero entries ranging from $-3n$ to $3n$, where $n$ is the dimension of the matrix.  

\begin{table}[!htb]
\centering
\begin{tabular}{|ccccc|}
\hline
\multicolumn{5}{|l|}{Runtime (in seconds) for Indefinite Matrices}                                   \\ \hline
\multicolumn{1}{|c|}{$n$}      & \multicolumn{1}{c|}{500}  & \multicolumn{1}{c|}{1000} & \multicolumn{1}{c|}{5000}  & 10000 \\ \hline
\multicolumn{1}{|c|}{CTA}  & \multicolumn{1}{c|}{0.34} & \multicolumn{1}{c|}{0.75} & \multicolumn{1}{c|}{9.38}  & 24.44 \\ \hline
\multicolumn{1}{|c|}{GMRES} & \multicolumn{1}{c|}{0.62} & \multicolumn{1}{c|}{1.55} & \multicolumn{1}{c|}{13.09} & 42.31 \\ \hline
\end{tabular}
\caption{Comparing runtime (in seconds) with CTA vs. GMRES for indefinite matrices for various dimensions $n$.}
\label{ind_fr1}
\end{table}

\begin{table}[!htb]
\centering
\begin{tabular}{|ccccc|}
\hline
\multicolumn{5}{|l|}{Number of Iterations for Indefinite Matrices}                                                             \\ \hline
\multicolumn{1}{|c|}{$n$}      & \multicolumn{1}{c|}{500} & \multicolumn{1}{c|}{1000} & \multicolumn{1}{c|}{5000} & 10000 \\ \hline
\multicolumn{1}{|c|}{CTA}  & \multicolumn{1}{c|}{456} & \multicolumn{1}{c|}{501}  & \multicolumn{1}{c|}{2332} & 5101  \\ \hline
\multicolumn{1}{|c|}{GMRES}   & \multicolumn{1}{c|}{490} & \multicolumn{1}{c|}{985}  & \multicolumn{1}{c|}{4788} & 9987    \\ \hline
\end{tabular}
\caption{Comparing number of iterations with CTA vs. GMRES for indefinite matrices for various dimensions $n$.}
\label{ind_fr2}
\end{table}

\begin{table}[!htb]
\centering
\begin{tabular}{|ccccc|}
\hline
\multicolumn{5}{|c|}{Quality of Solutions for Indefinite Matrices}                                                                                                           \\ \hline
\multicolumn{1}{|c|}{$n$}      & \multicolumn{1}{c|}{500}                           & \multicolumn{1}{c|}{1000}                          & \multicolumn{1}{c|}{5000}                          & 10000                         \\ \hline
\multicolumn{1}{|c|}{CTA}   & \multicolumn{1}{c|}{\num{1.0e-15}} & \multicolumn{1}{c|}{\num{1.0e-15}} & \multicolumn{1}{c|}{\num{1.0e-15}} & \num{1.0e-15} \\ \hline
\multicolumn{1}{|c|}{GMRES}     & \multicolumn{1}{c|}{\num{1.0e-15}} & \multicolumn{1}{c|}{\num{1.1e-15}} & \multicolumn{1}{c|}{\num{1.2e-15}} & \num{1.8e-15} \\ \hline
\end{tabular}
\caption{Comparing quality of solutions with CTA vs. GMRES for indefinite matrices for various dimensions $n$.}
\label{ind_fr3}
\end{table}

As noted by Tables \ref{ind_fr1}, \ref{ind_fr2}, and \ref{ind_fr3}, the results favor CTA considerably. For small dimensions $n \leq 500$, CTA is at least 39\% faster than GMRES. For large dimensions $n = 10000$, CTA is 46\% faster than GMRES. Compared to positive definite and positive semidefinite cases, the quality of solutions for GMRES in the indefinite case is worse, but uses more iterations. The quality of solutions for GMRES only decreases to an order of $1.8 \times 10^{-15}$ for dimension 10000, whereas CTA can be used for any arbitrary precision. 

\subsubsection{Poisson Equation with Dirichlet Boundary Conditions} \label{SecPSD_2}
Suppose $\Omega = (0, 1) \times (0, 1)$ is the unit square. Denote $\Gamma = \partial\Omega$ as the boundary of $\Omega$. The Poisson equation with Dirichlet boundary conditions is given as follows:
\[
-\triangle u = -\left(\frac{\partial^2 u}{\partial x^2} + \frac{\partial^2 u}{\partial y^2}\right) = f \text{ in } \Omega, \quad u = g \text{ on } \Gamma = \partial\Omega.
\]

Solving this setting by discretization using finite difference methods on a 4-point stencil yields the linear system $Au = f$, where $A$ is positive definite. Details on how $A$ is generated as well as the code to generate $A$ are given in Chen \cite{poisson}.

\begin{table}[!htb]
\centering
\begin{tabular}{|ccccc|}
\hline
\multicolumn{5}{|l|}{Runtime (in seconds) for PE Matrices}                                   \\ \hline
\multicolumn{1}{|c|}{$n$}      & \multicolumn{1}{c|}{500}  & \multicolumn{1}{c|}{1000} & \multicolumn{1}{c|}{5000}  & 10000 \\ \hline
\multicolumn{1}{|c|}{CTA}  & \multicolumn{1}{c|}{0.30} & \multicolumn{1}{c|}{0.73} & \multicolumn{1}{c|}{8.01}  & 22.22 \\ \hline
\multicolumn{1}{|c|}{GMRES}  & \multicolumn{1}{c|}{0.44} & \multicolumn{1}{c|}{0.87} & \multicolumn{1}{c|}{10.72} & 37.43 \\ \hline
\end{tabular}
\caption{Comparing runtime (in seconds) with CTA vs. GMRES for positive definite matrices generated for the Poisson Equation with Dirichlet Boundary Conditions for various dimensions $n$.}
\label{psd_ped1}
\end{table}

\begin{table}[!htb]
\centering
\begin{tabular}{|ccccc|}
\hline
\multicolumn{5}{|l|}{Number of Iterations for PE Matrices}                                                             \\ \hline
\multicolumn{1}{|c|}{$n$}      & \multicolumn{1}{c|}{500} & \multicolumn{1}{c|}{1000} & \multicolumn{1}{c|}{5000} & 10000 \\ \hline
\multicolumn{1}{|c|}{CTA}    & \multicolumn{1}{c|}{452} & \multicolumn{1}{c|}{501}  & \multicolumn{1}{c|}{2332} & 5099  \\ \hline
\multicolumn{1}{|c|}{GMRES}    & \multicolumn{1}{c|}{431} & \multicolumn{1}{c|}{917}  & \multicolumn{1}{c|}{4427} & 7199  \\ \hline
\end{tabular}
\caption{Comparing number of iterations with CTA vs. GMRES for positive definite matrices generated for the Poisson Equation with Dirichlet Boundary Conditions for various dimensions $n$.}
\label{psd_ped2}
\end{table}

\begin{table}[!htb]
\centering
\begin{tabular}{|ccccc|}
\hline
\multicolumn{5}{|c|}{Quality of Solutions for PE Matrices}                                                                                                           \\ \hline
\multicolumn{1}{|c|}{$n$}      & \multicolumn{1}{c|}{500}                           & \multicolumn{1}{c|}{1000}                          & \multicolumn{1}{c|}{5000}                          & 10000                         \\ \hline
\multicolumn{1}{|c|}{CTA}   & \multicolumn{1}{c|}{\num{1.0e-15}} & \multicolumn{1}{c|}{\num{2.1e-15}} & \multicolumn{1}{c|}{\num{9.5e-15}} & \num{9.6e-14} \\ \hline
\multicolumn{1}{|c|}{GMRES}   & \multicolumn{1}{c|}{\num{1.1e-15}} & \multicolumn{1}{c|}{\num{3.4e-15}} & \multicolumn{1}{c|}{\num{1.1e-14}} & \num{1.3e-13} \\ \hline
\end{tabular}
\caption{Comparing quality of solutions with CTA vs. GMRES for positive definite matrices generated for the Poisson Equation with Dirichlet Boundary Conditions for various dimensions $n$.}
\label{psd_ped3}
\end{table}

As noted by Tables \ref{psd_ped1}, \ref{psd_ped2}, and \ref{psd_ped3}, the results show strong favorability for CTA over GMRES in the case of Poisson equation with Dirichlet boundary conditions. For small dimensions $n \leq 500$, CTA is at least 32\% faster than GMRES. For large dimensions $n = 10000$, CTA is 41\% faster than GMRES. The quality of solutions for GMRES decreases to an order of $10^{-13}$ for dimension 10000, whereas CTA can be used for any arbitrary precision. 

\subsubsection{Poisson Equation with Neumann Boundary Conditions} \label{SecPSD_3}
Suppose $\Omega$ and $\Gamma$ are defined as in the previous section, with $n$ being the dimension. The Poisson equation with Neumann boundary conditions is given as follows:
\[
-\triangle u = -\left(\frac{\partial^2 u}{\partial x^2} + \frac{\partial^2 u}{\partial y^2}\right) = f \text{ in } \Omega, \quad \frac{\partial u}{\partial n} = g \text{ on } \Gamma.
\]

Solving this setting by discretization using finite difference methods on a 4-point stencil yields the linear system $Au = f$, where $A$ is positive semi-definite. This case is different from the Dirichlet boundary conditions setting, where a more accurate approximation of $\frac{\partial u}{\partial n}$ requires careful consideration. Details on how $A$ is generated as well as the code to generate $A$ are given in Chen \cite{poisson}.

\begin{table}[!htb]
\centering
\begin{tabular}{|ccccc|}
\hline
\multicolumn{5}{|l|}{Runtime (in seconds) for PE Matrices}                                   \\ \hline
\multicolumn{1}{|c|}{$n$}      & \multicolumn{1}{c|}{500}  & \multicolumn{1}{c|}{1000} & \multicolumn{1}{c|}{5000}  & 10000 \\ \hline
\multicolumn{1}{|c|}{CTA}  & \multicolumn{1}{c|}{0.30} & \multicolumn{1}{c|}{0.73} & \multicolumn{1}{c|}{8.03}  & 22.30 \\ \hline
\multicolumn{1}{|c|}{GMRES} & \multicolumn{1}{c|}{0.45} & \multicolumn{1}{c|}{0.88} & \multicolumn{1}{c|}{10.82} & 37.55 \\ \hline
\end{tabular}
\caption{Comparing runtime (in seconds) with CTA vs. GMRES for positive semi-definite matrices generated for the Poisson Equation with Neumann Boundary Conditions for various dimensions $n$.}
\label{psd_pen1}
\end{table}

\begin{table}[!htb]
\centering
\begin{tabular}{|ccccc|}
\hline
\multicolumn{5}{|l|}{Number of Iterations for PE Matrices}                                                             \\ \hline
\multicolumn{1}{|c|}{$n$}      & \multicolumn{1}{c|}{500} & \multicolumn{1}{c|}{1000} & \multicolumn{1}{c|}{5000} & 10000 \\ \hline
\multicolumn{1}{|c|}{CTA}  & \multicolumn{1}{c|}{451} & \multicolumn{1}{c|}{502}  & \multicolumn{1}{c|}{2344} & 5100  \\ \hline
\multicolumn{1}{|c|}{GMRES}    & \multicolumn{1}{c|}{433} & \multicolumn{1}{c|}{920}  & \multicolumn{1}{c|}{4435} & 7203  \\ \hline
\end{tabular}
\caption{Comparing number of iterations with CTA vs. GMRES for positive semi-definite matrices generated for the Poisson Equation with Neumann Boundary Conditions for various dimensions $n$.}
\label{psd_pen2}
\end{table}

\begin{table}[!htb]
\centering
\begin{tabular}{|ccccc|}
\hline
\multicolumn{5}{|c|}{Quality of Solutions for PE Matrices}                                                                                                           \\ \hline
\multicolumn{1}{|c|}{$n$}      & \multicolumn{1}{c|}{500}                           & \multicolumn{1}{c|}{1000}                          & \multicolumn{1}{c|}{5000}                          & 10000                         \\ \hline
\multicolumn{1}{|c|}{CTA}   & \multicolumn{1}{c|}{\num{1.0e-15}} & \multicolumn{1}{c|}{\num{2.2e-15}} & \multicolumn{1}{c|}{\num{9.6e-15}} & \num{9.7e-14} \\ \hline
\multicolumn{1}{|c|}{GMRES}   & \multicolumn{1}{c|}{\num{1.1e-15}} & \multicolumn{1}{c|}{\num{3.5e-15}} & \multicolumn{1}{c|}{\num{1.2e-14}} & \num{1.4e-13} \\ \hline
\end{tabular}
\caption{Comparing quality of solutions with CTA vs. GMRES for positive semi-definite matrices generated for the Poisson Equation with Neumann Boundary Conditions for various dimensions $n$.}
\label{psd_pen3}
\end{table}

As noted by Tables \ref{psd_pen1}, \ref{psd_pen2}, and \ref{psd_pen3}, the results show strong favorability for CTA over GMRES as well in the case of Poisson equation with Neumann boundary conditions. For small dimensions $n \leq 500$, CTA is at least 32\% faster than GMRES. For large dimensions $n = 10000$, CTA is 41\% faster than GMRES. The quality of solutions for GMRES decreases to an order of $10^{-13}$ for dimension 10000.

\subsubsection{Clement Tridiagonal Matrices} \label{SecTri_1}
Matrices generated by the Clement generator are tridiagonal matrices with zero diagonal. These matrices may be rank-deficient. More details on these matrices can also be found in Clement \cite{clement}.

\begin{table}[!htb]
\centering
\begin{tabular}{|ccccc|}
\hline
\multicolumn{5}{|l|}{Runtime (in seconds) for Clement Matrices}                                   \\ \hline
\multicolumn{1}{|c|}{$n$}      & \multicolumn{1}{c|}{500}  & \multicolumn{1}{c|}{1000} & \multicolumn{1}{c|}{5000}  & 10000 \\ \hline
\multicolumn{1}{|c|}{CTA}   & \multicolumn{1}{c|}{0.30} & \multicolumn{1}{c|}{0.69} & \multicolumn{1}{c|}{7.32}  & 21.63 \\ \hline
\multicolumn{1}{|c|}{GMRES}  & \multicolumn{1}{c|}{0.41} & \multicolumn{1}{c|}{0.85} & \multicolumn{1}{c|}{11.32} & 38.39 \\ \hline
\end{tabular}
\caption{Comparing runtime (in seconds) with CTA vs. GMRES for Clement tridiagonal matrices for various dimensions $n$.}
\label{tri_cl1}
\end{table}

\begin{table}[!htb]
\centering
\begin{tabular}{|ccccc|}
\hline
\multicolumn{5}{|l|}{Number of Iterations for Clement Matrices}                                                             \\ \hline
\multicolumn{1}{|c|}{$n$}      & \multicolumn{1}{c|}{500} & \multicolumn{1}{c|}{1000} & \multicolumn{1}{c|}{5000} & 10000 \\ \hline
\multicolumn{1}{|c|}{CTA}  & \multicolumn{1}{c|}{449} & \multicolumn{1}{c|}{511}  & \multicolumn{1}{c|}{2392} & 5117  \\ \hline
\multicolumn{1}{|c|}{GMRES}    & \multicolumn{1}{c|}{420} & \multicolumn{1}{c|}{920}  & \multicolumn{1}{c|}{4423} & 7192  \\ \hline
\end{tabular}
\caption{Comparing number of iterations with CTA vs. GMRES for Clement tridiagonal matrices for various dimensions $n$.}
\label{tri_cl2}
\end{table}

\begin{table}[!htb]
\centering
\begin{tabular}{|ccccc|}
\hline
\multicolumn{5}{|c|}{Quality of Solutions for Clement Matrices}                                                                                                           \\ \hline
\multicolumn{1}{|c|}{$n$}      & \multicolumn{1}{c|}{500}                           & \multicolumn{1}{c|}{1000}                          & \multicolumn{1}{c|}{5000}                          & 10000                         \\ \hline
\multicolumn{1}{|c|}{CTA}    & \multicolumn{1}{c|}{\num{1.2e-15}} & \multicolumn{1}{c|}{\num{2.2e-15}} & \multicolumn{1}{c|}{\num{9.8e-15}} & \num{9.8e-14} \\ \hline
\multicolumn{1}{|c|}{GMRES}   & \multicolumn{1}{c|}{\num{1.4e-15}} & \multicolumn{1}{c|}{\num{4.9e-15}} & \multicolumn{1}{c|}{\num{1.4e-14}} & \num{1.5e-13} \\ \hline
\end{tabular}
\caption{Comparing quality of solutions with CTA vs. GMRES for Clement tridiagonal matrices for various dimensions $n$.}
\label{tri_cl3}
\end{table}

As noted by Tables \ref{tri_cl1}, \ref{tri_cl2}, and \ref{tri_cl3}, the results show strong favorability for CTA over GMRES in the case of Clement generated matrices, especially for matrices with large dimension. For small dimensions $n \leq 500$, CTA is at least 27\% faster than GMRES. For large dimensions $n = 10000$, CTA is 44\% faster than GMRES. Similar to the case for PSD matrices, the quality of solutions for GMRES decreases to an order of $10^{-13}$ for dimension 10000.

\subsubsection{Dorr Tridiagonal Matrices} \label{SecTri_2}
Matrices generated by the Dorr generator are diagonally dominant M-matrices which may be ill-conditioned. These matrices are also invertible. More details can be found in Dorr \cite{dorr}.

\begin{table}[!htb]
\centering
\begin{tabular}{|ccccc|}
\hline
\multicolumn{5}{|l|}{Runtime (in seconds) for Dorr Matrices}                                   \\ \hline
\multicolumn{1}{|c|}{$n$}      & \multicolumn{1}{c|}{500}  & \multicolumn{1}{c|}{1000} & \multicolumn{1}{c|}{5000}  & 10000 \\ \hline
\multicolumn{1}{|c|}{CTA}    & \multicolumn{1}{c|}{0.32} & \multicolumn{1}{c|}{0.71} & \multicolumn{1}{c|}{8.95}  & 22.24 \\ \hline
\multicolumn{1}{|c|}{GMRES}   & \multicolumn{1}{c|}{0.42} & \multicolumn{1}{c|}{0.86} & \multicolumn{1}{c|}{11.44} & 39.61 \\ \hline
\end{tabular}
\caption{Comparing runtime (in seconds) with CTA vs. GMRES for Dorr tridiagonal matrices for various dimensions $n$.}
\label{tri_do1}
\end{table}

\begin{table}[!htb]
\centering
\begin{tabular}{|ccccc|}
\hline
\multicolumn{5}{|l|}{Number of Iterations for Dorr Matrices}                                                             \\ \hline
\multicolumn{1}{|c|}{$n$}      & \multicolumn{1}{c|}{500} & \multicolumn{1}{c|}{1000} & \multicolumn{1}{c|}{5000} & 10000 \\ \hline
\multicolumn{1}{|c|}{CTA}  & \multicolumn{1}{c|}{452} & \multicolumn{1}{c|}{517}  & \multicolumn{1}{c|}{2419} & 5172  \\ \hline
\multicolumn{1}{|c|}{GMRES}   & \multicolumn{1}{c|}{423} & \multicolumn{1}{c|}{924}  & \multicolumn{1}{c|}{4431} & 7210  \\ \hline
\end{tabular}
\caption{Comparing runtime (in seconds) with CTA vs. GMRES for Dorr tridiagonal matrices for various dimensions $n$.}
\label{tri_do2}
\end{table}

\begin{table}[!htb]
\centering
\begin{tabular}{|ccccc|}
\hline
\multicolumn{5}{|c|}{Quality of Solutions for Dorr Matrices}                                                                                                           \\ \hline
\multicolumn{1}{|c|}{$n$}      & \multicolumn{1}{c|}{500}                           & \multicolumn{1}{c|}{1000}                          & \multicolumn{1}{c|}{5000}                          & 10000                         \\ \hline
\multicolumn{1}{|c|}{CTA}   & \multicolumn{1}{c|}{\num{1.1e-15}} & \multicolumn{1}{c|}{\num{2.1e-15}} & \multicolumn{1}{c|}{\num{9.7e-15}} & \num{9.5e-14} \\ \hline
\multicolumn{1}{|c|}{GMRES}    & \multicolumn{1}{c|}{\num{1.3e-15}} & \multicolumn{1}{c|}{\num{4.8e-15}} & \multicolumn{1}{c|}{\num{1.5e-14}} & \num{1.5e-13} \\ \hline
\end{tabular}
\caption{Comparing quality of solutions with CTA vs. GMRES for Dorr tridiagonal matrices for various dimensions $n$.}
\label{tri_do3}
\end{table}

As with the Clement tridiagonal matrices, Tables \ref{tri_do1}, \ref{tri_do2}, and \ref{tri_do3} show strong favorability for CTA over GMRES in the case of Dorr generated matrices, especially for matrices with large dimension. For small dimensions $n \leq 500$, CTA is at least 31\% faster than GMRES. For large dimensions $n = 10000$, CTA is 44\% faster than GMRES. Similar to the case for Clement tridiagonal matrices, the quality of solutions for GMRES decreases to an order of $10^{-13}$ for dimension 10000.

\subsubsection{Matrices Generated by MVMMCD: Computational Fluid Dynamics} \label{SecTri_3}
In this section we generated test matrices using the generator MVMMCD for the following constant-coefficient convection diffusion equation:
\[
-\triangle u + 2p_1u_x + 2p_2u_y - p_3u = f \text{ in } \Omega, \quad u = g \text{ on } \Gamma,
\]
where $p_1$, $p_2$, and $p_3$ are positive constants and $\Omega = [0, 1] \times [0, 1]$ is the unit square with $\Gamma = \partial\Omega$ as the boundary of $\Omega$. Following a solution scheme similar to solving the Poisson equation in the previous section, we solve this equation by discretization using the finite difference methods with a 5-point stencil on an $n \times n$ grid yields the linear system $Au = b$, where $A$ is an $n^2 \times n^2$ block tridiagonal matrix and $u, b \in \mathbb{R}^{n^2}$. More details can be found in the NEP collection \cite{mvmnep}.

\begin{table}[!htb]
\centering
\begin{tabular}{|cccc|}
\hline
\multicolumn{4}{|c|}{Runtime (in sec) for MVMMCD}                                       \\ \hline
\multicolumn{1}{|c|}{$n^2$}       & \multicolumn{1}{c|}{100} & \multicolumn{1}{c|}{1600} & 4900  \\ \hline
\multicolumn{1}{|c|}{CTA}    & \multicolumn{1}{c|}{1.2} & \multicolumn{1}{c|}{4.57} & 47.11 \\ \hline
\multicolumn{1}{|c|}{GMRES}  & \multicolumn{1}{c|}{1.8} & \multicolumn{1}{c|}{5.28} & 48.27 \\ \hline
\end{tabular}
\caption{Comparing runtime (in seconds) with CTA vs. GMRES for matrices generated using MVMMCD for various dimensions $n$.}
\label{tri_fd1}
\end{table}

\begin{table}[!htb]
\centering
\begin{tabular}{|cccc|}
\hline
\multicolumn{4}{|c|}{Number of Iterations for MVMMCD}                                      \\ \hline
\multicolumn{1}{|c|}{$n^2$}       & \multicolumn{1}{c|}{100} & \multicolumn{1}{c|}{1600} & 4900 \\ \hline
\multicolumn{1}{|c|}{CTA}    & \multicolumn{1}{c|}{201} & \multicolumn{1}{c|}{831}  & 4102 \\ \hline
\multicolumn{1}{|c|}{GMRES}  & \multicolumn{1}{c|}{224} & \multicolumn{1}{c|}{969}  & 4839 \\ \hline
\end{tabular}
\caption{Comparing number of iterations with CTA vs. GMRES for matrices generated using MVMMCD for various dimensions $n$.}
\label{tri_fd2}
\end{table}

\begin{table}[!htb]
\centering
\begin{tabular}{|cccc|}
\hline
\multicolumn{4}{|c|}{Quality of Solutions for MVMMCD}                                                \\ \hline
\multicolumn{1}{|c|}{$n^2$}       & \multicolumn{1}{c|}{100}                           & \multicolumn{1}{c|}{1600}                          & 4900                          \\ \hline
\multicolumn{1}{|c|}{CTA}    & \multicolumn{1}{c|}{\num{1.8e-8}}                            & \multicolumn{1}{c|}{\num{1.7e-6}}                            & \num{3.4e-6}                            \\ \hline
\multicolumn{1}{|c|}{GMRES}  & \multicolumn{1}{c|}{NC} & \multicolumn{1}{c|}{NC} & NC \\ \hline
\end{tabular}
\caption{Comparing quality of solutions with CTA vs. GMRES for matrices generated using MVMMCD for various dimensions $n$. Here, "NC" stands for not converging.}
\label{tri_fd3}
\end{table}

Tables \ref{tri_fd1}, \ref{tri_fd2}, and \ref{tri_fd3} show strong favorability for CTA over GMRES for matrices generated by this Computational Fluid Dynamics problem. In particular, GMRES does not converge even for generated matrices with small dimension, whereas CTA can give solutions for precision up to $10^{-6}$ faster than GMRES can report non-convergence. For real world applications such as Computational Fluid Dynamics, it is imperative that a linear system solver provides a good solution in reasonable time. We have shown that CTA is successful in achieving this goal, while GMRES fails to do so, making it ineffective in this situation.

\subsubsection{Lotkin Matrices}  \label{SecRan_2}
Matrices generated by the Lotkin generator are non-symmetric, ill-conditioned matrices with small, negative eigenvalues. More details on these matrices can also be found in Lotkin \cite{lotkin}.

\begin{table}[!htb]
\centering
\begin{tabular}{|ccccc|}
\hline
\multicolumn{5}{|l|}{Runtime (in seconds) for Lotkin Matrices}                                   \\ \hline
\multicolumn{1}{|c|}{$n$}      & \multicolumn{1}{c|}{500}  & \multicolumn{1}{c|}{1000} & \multicolumn{1}{c|}{5000}  & 10000 \\ \hline
\multicolumn{1}{|c|}{CTA}   & \multicolumn{1}{c|}{1.27} & \multicolumn{1}{c|}{4.86} & \multicolumn{1}{c|}{49.48}  & 68.69 \\ \hline
\multicolumn{1}{|c|}{GMRES}  & \multicolumn{1}{c|}{1.34} & \multicolumn{1}{c|}{4.89} & \multicolumn{1}{c|}{48.97} & 69.52 \\ \hline
\end{tabular}
\caption{Comparing runtime (in seconds) with CTA vs. GMRES for Lotkin matrices for various dimensions $n$.}
\label{ran_lo1}
\end{table}

\begin{table}[!htb]
\centering
\begin{tabular}{|ccccc|}
\hline
\multicolumn{5}{|l|}{Number of Iterations for Lotkin Matrices}                                                             \\ \hline
\multicolumn{1}{|c|}{$n$}      & \multicolumn{1}{c|}{500} & \multicolumn{1}{c|}{1000} & \multicolumn{1}{c|}{5000} & 10000 \\ \hline
\multicolumn{1}{|c|}{CTA}  & \multicolumn{1}{c|}{237} & \multicolumn{1}{c|}{828}  & \multicolumn{1}{c|}{4328} & 7130  \\ \hline
\multicolumn{1}{|c|}{GMRES}    & \multicolumn{1}{c|}{240} & \multicolumn{1}{c|}{937}  & \multicolumn{1}{c|}{4495} & 7101  \\ \hline
\end{tabular}
\caption{Comparing runtime (in seconds) with CTA vs. GMRES for Lotkin matrices for various dimensions $n$.}
\label{ran_lo2}
\end{table}

\begin{table}[!htb]
\centering
\begin{tabular}{|ccccc|}
\hline
\multicolumn{5}{|c|}{Quality of Solutions for Lotkin Matrices}                                                                                                           \\ \hline
\multicolumn{1}{|c|}{$n$}      & \multicolumn{1}{c|}{500}                           & \multicolumn{1}{c|}{1000}                          & \multicolumn{1}{c|}{5000}                          & 10000                         \\ \hline
\multicolumn{1}{|c|}{CTA}    & \multicolumn{1}{c|}{\num{5.8e-8}} & \multicolumn{1}{c|}{\num{4.5e-7}} & \multicolumn{1}{c|}{\num{1.2e-6}} & \num{1.1e-6} \\ \hline
\multicolumn{1}{|c|}{GMRES}   & \multicolumn{1}{c|}{\num{5.9e-8}} & \multicolumn{1}{c|}{\num{3.2e-7}} & \multicolumn{1}{c|}{\num{1.0e-6}} & \num{1.2e-6} \\ \hline
\end{tabular}
\caption{Comparing quality of solutions with CTA vs. GMRES for Lotkin matrices for various dimensions $n$.}
\label{ran_lo3}
\end{table}

As noted by Tables \ref{ran_lo1}, \ref{ran_lo2}, and \ref{ran_lo3}, the results show strong favorability for CTA over GMRES in the case of Lotkin generated matrices. For small dimensions $n \leq 500$, CTA is at least 5\% faster than GMRES. For large dimensions $n = 10000$, CTA is about 1\% faster than GMRES. The quality of solutions for GMRES decreases to an order of $10^{-6}$ for dimension 10000.

\subsubsection{Ordinary Differential Equations}
Consider the following second order ODE:
\begin{align} \label{ODEForm1}
y'' = -\mu^2y,
\end{align}
with boundary conditions
\[
y(0) = 0 \text{ and } y'(0) + \gamma y'(1) = 0, \quad 0 < \gamma < 1.
\]
The eigenproblem of (\ref{ODEForm1}) can be approximated by finite differences by letting $y_i$ be the approximate solutions at $x_i = i / (n + 1)$ for $0 \leq i \leq n$. Then, this problem can be recast as follows:
\[
Ay = -\mu^2By, 
\]
where $B = h^2\text{diag}(1,1,\cdots,1,0)$, and $h$ is a parameter. For these computational results, we use the matrices $M = A^{-1}B$ and solve $Mx = b$, where $b$ is the row-sum vector of $M$. More details can be found in the NEP collection \cite{mvmnep}.
\begin{table}[!htb]
\centering
\begin{tabular}{|cccc|}
\hline
\multicolumn{4}{|c|}{Runtime (in sec) for MVMODE}                                       \\ \hline
\multicolumn{1}{|c|}{$n$}       & \multicolumn{1}{c|}{100} & \multicolumn{1}{c|}{1000} & 5000  \\ \hline
\multicolumn{1}{|c|}{CTA}    & \multicolumn{1}{c|}{1.3} & \multicolumn{1}{c|}{4.48} & 46.29 \\ \hline
\multicolumn{1}{|c|}{GMRES}  & \multicolumn{1}{c|}{1.8} & \multicolumn{1}{c|}{5.32} & 49.19 \\ \hline
\end{tabular}
\caption{Comparing runtime (in seconds) with CTA vs. GMRES for matrices generated using MVMODE for various dimensions $n$.}
\label{ran_od1}
\end{table}

\begin{table}[!htb]
\centering
\begin{tabular}{|cccc|}
\hline
\multicolumn{4}{|c|}{Number of Iterations for MVMODE}                                      \\ \hline
\multicolumn{1}{|c|}{$n$}       & \multicolumn{1}{c|}{100} & \multicolumn{1}{c|}{1000} & 5000 \\ \hline
\multicolumn{1}{|c|}{CTA}    & \multicolumn{1}{c|}{203} & \multicolumn{1}{c|}{97}  & 4012 \\ \hline
\multicolumn{1}{|c|}{GMRES}  & \multicolumn{1}{c|}{219} & \multicolumn{1}{c|}{831}  & 4928 \\ \hline
\end{tabular}
\caption{Comparing number of iterations with CTA vs. GMRES for matrices generated using MVMODE for various dimensions $n$.}
\label{ran_od2}
\end{table}

\begin{table}[!htb]
\centering
\begin{tabular}{|cccc|}
\hline
\multicolumn{4}{|c|}{Quality of Solutions for MVMODE}                                                \\ \hline
\multicolumn{1}{|c|}{$n$}       & \multicolumn{1}{c|}{100}                           & \multicolumn{1}{c|}{1000}                          & 5000                          \\ \hline
\multicolumn{1}{|c|}{CTA}    & \multicolumn{1}{c|}{\num{1.8e-8}}                            & \multicolumn{1}{c|}{\num{1.6e-6}}                            & \num{3.2e-6}                            \\ \hline
\multicolumn{1}{|c|}{GMRES}  & \multicolumn{1}{c|}{NC} & \multicolumn{1}{c|}{NC} & NC \\ \hline
\end{tabular}
\caption{Comparing quality of solutions with CTA vs. GMRES for matrices generated using MVMODE for various dimensions $n$. Here, "NC" stands for not converging.}
\label{ran_od3}
\end{table}

As in the Computational Fluid Dynamics section, Tables \ref{ran_od1}, \ref{ran_od2}, and \ref{ran_od3} show strong favorability for CTA over GMRES for matrices generated by this problem regarding ODEs. In particular, GMRES does not converge even for generated matrices with small dimension, whereas CTA can give solutions for precision up to $10^{-6}$ faster than GMRES can report non-convergence. Even though the matrices are ill-conditioned, we have shown that CTA is successful in providing a good solution even for problems in ODEs while GMRES fails to do so. 

\subsubsection{Results for MatrixMarket \& SuiteSparse Matrices}
In this section, we present computational results regarding a data of 3629 general real world matrices taken from two sources: the MatrixMarket and SuiteSparse collections \cite{matrixmarket, suitesparse}. When solving for $Ax=b$, GMRES does not converge to a solution for many matrices in this collection. Therefore, we split this section into two parts: 1) when GMRES converges,  and 2) when GMRES does not converge.

\begin{itemize}
    \item {\bf Case 1: GMRES: Non-convergent Cases}
    
    Out of the 3629 matrices, there are 1468 of them for which GMRES does not converge. So, about 40.5\% of the time, GMRES does not converge. In this case, CTA runs approximately 8\% faster than GMRES on average, and uses 3\% less iterations. The relative residual for CTA is approximately $10^{-6}$ on average.
    
    \item {\bf Case 2: GMRES: Convergent Cases}
    
    Out of the 3629 matrices, there are 2161 of them for which GMRES does converge. In this case, CTA runs approximately 2\% slower than GMRES on average, and uses 3\% less iterations. The relative residual for CTA and GMRES is approximately $10^{-6}$ on average, with an average difference of $\num{1.3e-7}$.
\end{itemize}

\subsection{Comparison of Performance between CTA and GMRES when Less Precision is Necessary}
In some cases, practitioners who want a solution of reasonable quality might not need precision up to $10^{-15}$. We show comparison of CTA and GMRES for precisions $10^{-15}$, $10^{-10}$, and $10^{-5}$. These computational results also favor CTA.

\subsubsection{Comparison of Performance for PSD Matrices}
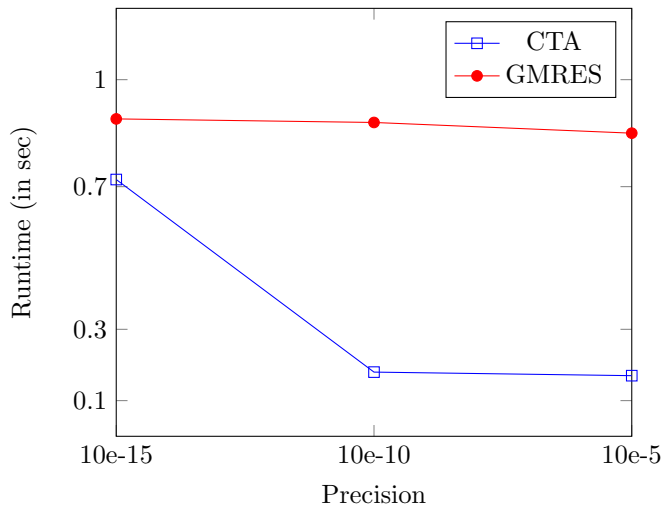
\begin{figure}[!htb]
\centering
\begin{tikzpicture}
\begin{axis}[
    title={Runtimes in Different Precisions for PSD Matrices with $n=1000$},
    xlabel={Precision},
    ylabel={Runtime (in sec)},
    xmin=1, xmax=3,
    ymin=0, ymax=1.2,
    xtick={1,2,3},
    xticklabels={10e-15, 10e-10, 10e-5},
    ytick={.1,.3,.7,1},
    legend pos=north east,
]

\addplot[
    color=blue,
    mark=square,
    ]
    coordinates {
    (1,.72)(2,.18)(3,.17)
    };
    \addlegendentry{CTA}

\addplot[
    color=red,
    mark=*,
    ]
    coordinates {
    (1,.89)(2,.88)(3,.85)
    };
    \addlegendentry{GMRES}
    
\end{axis}
\end{tikzpicture}
\caption{Comparison of runtimes between CTA and GMRES for full-rank PSD matrices with dimension $n = 1000$. The x-axis gives precisions, and y-axis give runtimes in seconds.}
\label{fig_psd1}
\end{figure}

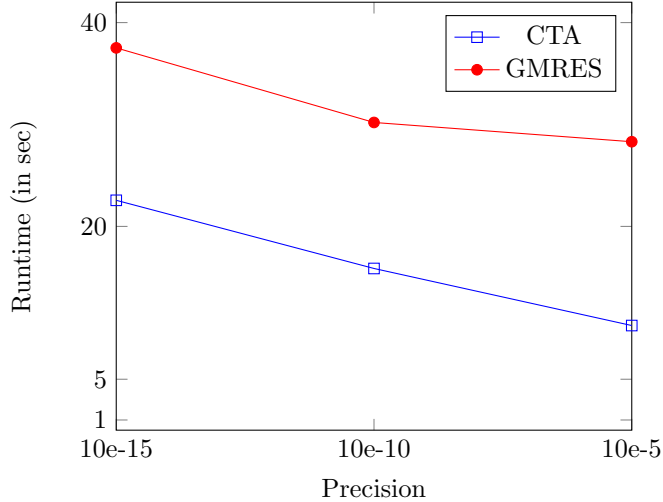
\begin{figure}[!htb]
\centering
\begin{tikzpicture}
\begin{axis}[
    title={Runtimes in Different Precisions for PSD Matrices with $n=10000$},
    xlabel={Precision},
    ylabel={Runtime (in sec)},
    xmin=1, xmax=3,
    ymin=0, ymax=42,
    xtick={1,2,3},
    xticklabels={10e-15, 10e-10, 10e-5},
    ytick={1,5,20,40},
    legend pos=north east,
]

\addplot[
    color=blue,
    mark=square,
    ]
    coordinates {
    (1,22.56)(2,15.87)(3,10.27)
    };
    \addlegendentry{CTA}

\addplot[
    color=red,
    mark=*,
    ]
    coordinates {
    (1,37.52)(2,30.20)(3,28.31)
    };
    \addlegendentry{GMRES}
    
\end{axis}
\end{tikzpicture}
\caption{Comparison of runtimes between CTA and GMRES for full-rank PSD matrices with dimension $n = 10000$. The x-axis gives precisions, and y-axis give runtimes in seconds.}
\label{fig_psd2}
\end{figure}

As shown by Figures \ref{fig_psd1} and \ref{fig_psd2}, the case for PSD matrices highlights the difference in runtimes between CTA and GMRES. In the case where less precision is needed, such as $10^{-10}$, CTA achieves an 80\% speed-up in runtime for dimension $n = 1000$. In the case where the matrix is larger, e.g. $n = 10000$, CTA achieves a 47\% speed-up in runtime. 

\subsubsection{Comparison of Performance for Clement Tridiagonal Matrices}
\begin{figure}[!htb]
\centering
\begin{tikzpicture}
\begin{axis}[
    title={Runtimes in Different Precisions for Clement Matrices with $n=1000$},
    xlabel={Precision},
    ylabel={Runtime (in sec)},
    xmin=1, xmax=3,
    ymin=0, ymax=1.2,
    xtick={1,2,3},
    xticklabels={10e-15, 10e-10, 10e-5},
    ytick={.1,.3,.7,1},
    legend pos=north east,
]

\addplot[
    color=blue,
    mark=square,
    ]
    coordinates {
    (1,.69)(2,.37)(3,.33)
    };
    \addlegendentry{CTA}

\addplot[
    color=red,
    mark=*,
    ]
    coordinates {
    (1,.85)(2,.79)(3,.70)
    };
    \addlegendentry{GMRES}
    
\end{axis}
\end{tikzpicture}
\caption{Comparison of runtimes between CTA and GMRES for Clement tridiagonal matrices with dimension $n = 1000$. The x-axis gives precisions, and y-axis give runtimes in seconds.}
\label{fig_cle1}
\end{figure}
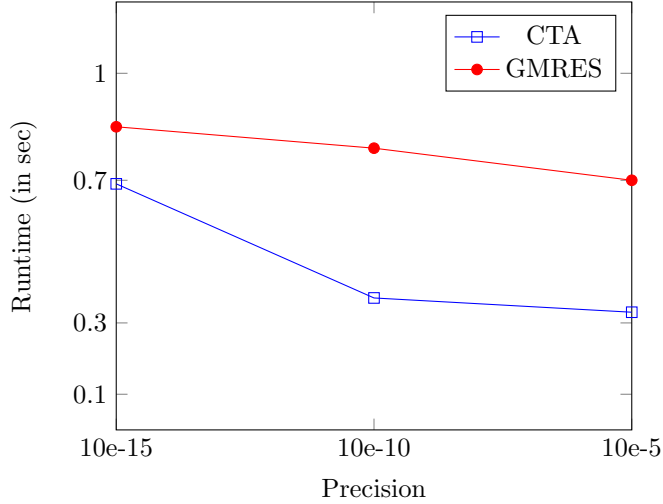

\begin{figure}[!htb]
\centering
\begin{tikzpicture}
\begin{axis}[
    title={Runtimes in Different Precisions for Clement Matrices with $n=10000$},
    xlabel={Precision},
    ylabel={Runtime (in sec)},
    xmin=1, xmax=3,
    ymin=0, ymax=50,
    xtick={1,2,3},
    xticklabels={10e-15, 10e-10, 10e-5},
    ytick={1,5,20,40},
    legend pos=north east,
]

\addplot[
    color=blue,
    mark=square,
    ]
    coordinates {
    (1,21.63)(2,15.32)(3,10.65)
    };
    \addlegendentry{CTA}

\addplot[
    color=red,
    mark=*,
    ]
    coordinates {
    (1,38.39)(2,36.92)(3,35.17)
    };
    \addlegendentry{GMRES}
    
\end{axis}
\end{tikzpicture}
\caption{Comparison of runtimes between CTA and GMRES for Clement tridiagonal matrices with dimension $n = 10000$. The x-axis gives precisions, and y-axis give runtimes in seconds.}
\label{fig_cle2}
\end{figure}
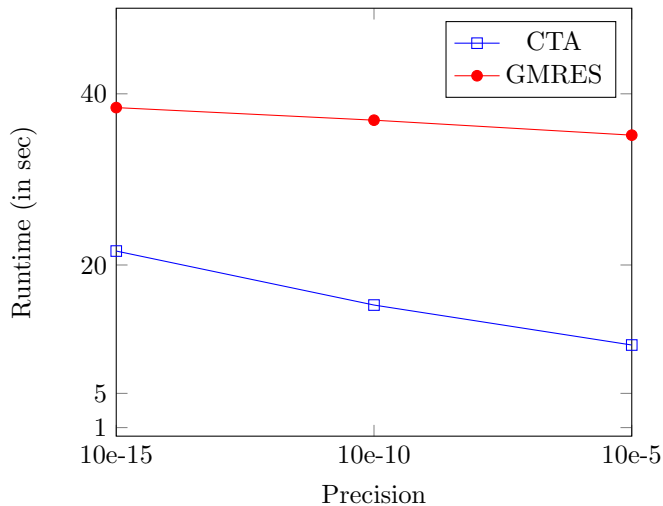

As shown by Figures \ref{fig_cle1} and \ref{fig_cle2}, the case for Clement tridiagonal matrices is similar to the case of the PSD matrices in regards to the difference in runtimes between CTA and GMRES. In the case where less precision is needed, such as $10^{-10}$, CTA achieves a 53\% speed-up in runtime for dimension $n = 1000$. In the case where the matrix is larger, e.g. $n = 10000$, CTA achieves a 59\% speed-up in runtime. 

\subsection{Computational Results for Rectangular Matrices} \label{SecRecMat}
In this section, we present computational results for rectangular matrices, i.e. $A$ is $m \times n$ where $m \not= n$, showing our hybrid CTA/TA algorithm outperforms ``lsqminnorm", the current state-of-the-art algorithm for finding minimum-norm least squares solution. The 138 matrices we used for this section are taken from the dataset ``lpnetlib" in the SuiteSparse collection \cite{suitesparse}. The matrix dimensions range from as small as $70 \times 100$ to as large as $10000 \times 10000000$. Due to the range in sizes, we split the dataset into three classes in terms of the number of rows $m$: $100$ (with $n$ ranging from $200$ to $1000$), $1000$ (with $n$ ranging from $2000$ to $10000$), and $10000$ (with $n$ ranging from $20000$ to $1000000$.

In our experiments, we solve $Ax = b$ where $A$ is $m \leq n$, where the system can be feasible or infeasible, and $b$ is the row-sum vector of $A$, therefore $x = [1, \cdots, 1]^T$ is a solution. We compare our own hybrid implementations of CTA and TA algorithm to Matlab's implementation ``lsqminnorm", the state-of-the-art algorithm for solving rectangular systems. More details of Matlab's implementation of ``lsqminnorm" can be found in the documentation for Matlab. Our two hybrid implementations of CTA/TA first runs CTA as in Section \ref{SecRecMat}. The first hybrid algorithm runs CTA with the parameter $\varepsilon = 10^{-8}$ and the second hybrid algorithm with parameter $\varepsilon = 10^{-15}$. Then, both hybrid algorithms runs TA, as described in Algorithm with parameter $\varepsilon = 10^{-15}$. As in the case of Section \ref{SecRecMat}, we terminate our hybrid algorithms early once the quality of solutions of the hybrid algorithms are better than the solution given by ``lsqminnorm". When $A$ is large and sparse, as in the case of square matrices, sparse matrix computation is used. If $x_\varepsilon$ is the $\varepsilon$-approximate solution to $Ax = b$, the inputs for TA is simply $A$ and $b$. Otherwise, the inputs for TA, are $A^TA$ and $A^Tb$, used implicitly via matrix-vector multiplications. 

\subsubsection{Feasible Rectangular Systems}
In this section, we compare our hybrid CTA/TA algorithms against ``lsqminnorm" both in runtimes and quality of solutions for feasible rectangular systems (FRS).

\begin{table}[!htb]
\centering
\begin{tabular}{|cccc|}
\hline
\multicolumn{4}{|c|}{Average Runtime (in seconds) for FRS}                                   \\ \hline
\multicolumn{1}{|c|}{$n$}      & \multicolumn{1}{c|}{100}  & \multicolumn{1}{c|}{1000} & 10000 \\ \hline
\multicolumn{1}{|c|}{$m$}      & \multicolumn{1}{c|}{200 - 1000}  & \multicolumn{1}{c|}{2000 - 10000} & 20000 - 1000000 \\ \hline
\multicolumn{1}{|c|}{CTA/TA ($\varepsilon = 10^{-8}$)}  & \multicolumn{1}{c|}{1.2} & \multicolumn{1}{c|}{40.17} & 421.28 \\ \hline
\multicolumn{1}{|c|}{CTA/TA ($\varepsilon = 10^{-15}$)}  & \multicolumn{1}{c|}{1.3} & \multicolumn{1}{c|}{42.29} & 430.35 \\ \hline
\multicolumn{1}{|c|}{lsqminnorm} & \multicolumn{1}{c|}{1.87} & \multicolumn{1}{c|}{62.78} & 612.99 \\ \hline
\end{tabular}
\caption{Comparing average runtimes (in seconds) with our hybrid algorithms vs. lsqminnorm for feasible rectangular systems for various dimensions $n$.}
\label{feas_rec_times}
\end{table}

\begin{table}[!htb]
\centering
\begin{tabular}{|cccc|}
\hline
\multicolumn{4}{|c|}{Quality of Solutions for FRS}                                   \\ \hline
\multicolumn{1}{|c|}{$n$}      & \multicolumn{1}{c|}{100}  & \multicolumn{1}{c|}{1000} & 10000 \\ \hline
\multicolumn{1}{|c|}{$m$}      & \multicolumn{1}{c|}{200 - 1000}  & \multicolumn{1}{c|}{2000 - 10000} & 20000 - 1000000 \\ \hline
\multicolumn{1}{|c|}{CTA/TA ($\varepsilon = 10^{-8}$)}  & \multicolumn{1}{c|}{9.9e-16} & \multicolumn{1}{c|}{9.9e-16} & 9.9e-16 \\ \hline
\multicolumn{1}{|c|}{CTA/TA ($\varepsilon = 10^{-15}$)}  & \multicolumn{1}{c|}{9.9e-16} & \multicolumn{1}{c|}{9.9e-16} & 1.0e-15 \\ \hline
\multicolumn{1}{|c|}{lsqminnorm} & \multicolumn{1}{c|}{1.1e-15} & \multicolumn{1}{c|}{1.2e-15} & 1.2e-15 \\ \hline
\end{tabular}
\caption{Comparing quality of solutions with our hybrid CTA/TA algorithms vs. lsqminnorm for feasible rectangular systems for various dimensions $n$.}
\label{feas_rec_qual}
\end{table}

As shown in Table \ref{feas_rec_times}, and Table \ref{feas_rec_qual}, our hybrid CTA/TA algorithms outperforms ``lsqminnorm" both in speed and quality of solutions. In the case where $n$ is around $10000$, our hybrid CTA/TA algorithms runs about 45.5\% faster than ``lsqminnorm". 

\subsubsection{Infeasible Rectangular Systems}
In this section, we compare our hybrid CTA/TA algorithms against ``lsqminnorm" both in runtimes and quality of solutions for infeasible rectangular systems (IFRS).

\begin{table}[!htb]
\centering
\begin{tabular}{|cccc|}
\hline
\multicolumn{4}{|c|}{Average Runtime (in seconds) for IFRS}                                   \\ \hline
\multicolumn{1}{|c|}{$n$}      & \multicolumn{1}{c|}{100}  & \multicolumn{1}{c|}{1000} & 10000 \\ \hline
\multicolumn{1}{|c|}{$m$}      & \multicolumn{1}{c|}{200 - 1000}  & \multicolumn{1}{c|}{2000 - 10000} & 20000 - 1000000 \\ \hline
\multicolumn{1}{|c|}{CTA/TA ($\varepsilon = 10^{-8}$)}  & \multicolumn{1}{c|}{1.8} & \multicolumn{1}{c|}{52.20} & 481.58 \\ \hline
\multicolumn{1}{|c|}{CTA/TA ($\varepsilon = 10^{-15}$)}  & \multicolumn{1}{c|}{1.9} & \multicolumn{1}{c|}{54.64} & 492.21 \\ \hline
\multicolumn{1}{|c|}{lsqminnorm} & \multicolumn{1}{c|}{2.73} & \multicolumn{1}{c|}{79.25} & 703.34 \\ \hline
\end{tabular}
\caption{Comparing average runtimes (in seconds) with our hybrid algorithms vs. lsqminnorm for feasible rectangular systems for various dimensions $n$.}
\label{ifeas_rec_times}
\end{table}

\begin{table}[!htb]
\centering
\begin{tabular}{|cccc|}
\hline
\multicolumn{4}{|c|}{Quality of Solutions for IFRS}                                   \\ \hline
\multicolumn{1}{|c|}{$n$}      & \multicolumn{1}{c|}{100}  & \multicolumn{1}{c|}{1000} & 10000 \\ \hline
\multicolumn{1}{|c|}{$m$}      & \multicolumn{1}{c|}{200 - 1000}  & \multicolumn{1}{c|}{2000 - 10000} & 20000 - 1000000 \\ \hline
\multicolumn{1}{|c|}{CTA/TA ($\varepsilon = 10^{-8}$)}  & \multicolumn{1}{c|}{9.8e-16} & \multicolumn{1}{c|}{9.8e-16} & 9.9e-16 \\ \hline
\multicolumn{1}{|c|}{CTA/TA ($\varepsilon = 10^{-15}$)}  & \multicolumn{1}{c|}{9.8e-16} & \multicolumn{1}{c|}{9.8e-16} & 1.0e-15 \\ \hline
\multicolumn{1}{|c|}{lsqminnorm} & \multicolumn{1}{c|}{1.2e-15} & \multicolumn{1}{c|}{1.2e-15} & 1.2e-15 \\ \hline
\end{tabular}
\caption{Comparing quality of solutions with our hybrid CTA/TA algorithms vs. lsqminnorm for feasible rectangular systems for various dimensions $n$.}
\label{ifeas_rec_qual}
\end{table}

As shown in Table \ref{ifeas_rec_times}, and Table \ref{ifeas_rec_qual}, our hybrid CTA/TA algorithms outperforms ``lsqminnorm" both in speed and quality of solutions. In the case where $n$ is around $10000$, our hybrid CTA/TA algorithms runs about 45.5\% faster than ``lsqminnorm". 

\section{Extension of TA to Linear Programming Feasibility} \label{sec6}

Linear programming (LP) feasibility problem is testing the solvability of a linear system $Ax=b$ when the nonnegativity constraints $x \geq 0$ are also added. Specifically, we wish to test the feasibility of the set,
\begin{equation}  \label{vrend1}
\Omega =\{x: Ax=b, \quad x \geq 0\}.
\end{equation}
Thus in  LP feasibility we seek a special solution of the linear system that lies in the nonnegativity cone. Theoretically speaking, while the {\it fully polynomial-time} solvability of testing if a linear system is accomplished via Gaussian elimination, the {\it polynomial-time} solvability of LP was first established by Khachiyan \cite{Khach80}. While linear system can be solved in time polynomial in the dimension of the problem, LP solvability depends polynomially in the dimension of the problem and logarithmically  in the size of the input data. In practice, both problems are significant and closely related.

Just as we have the linear system $Ax=b$ and the corresponding normal equation $A^TAx=A^Tb$, in linear programming we can consider solving $Ax=b$, $x \geq 0$ and a corresponding normal version of that, $A^TAx=A^Tb$, $x \geq 0$. The later problem is not considered or spoken of in linear programming literature. However, since we will be attempting to solve LP feasibility via the Triangle Algorithm we will need 
to consider both problems. Thus we first state a simple relationship between the two. 

\begin{prop} Suppose $Ax=b, x\geq 0$ is solvable.  If $\overline x$ is a solution to $A^TAx=A^Tb, x \geq 0$, then $\overline x$ is also a solution to $Ax=b$.
\end{prop}

\begin{proof} From Lagrange multiplier optimality condition it follows that 
$\overline x$ must also be a solution to $\min \{|Ax - b |:  x \geq 0\}$. 
\end{proof}

\begin{remark}  There are cases when $Ax=b, x \geq 0$ is infeasible but $A^TAx=A^Tb, x \geq 0$ is feasible. To give an example, take an case where $Ax=b$ is infeasible but $A^TAx=nA^Tb$ is feasible, say $x'$ is a solution. Now write $x'=x_1-x_2$ where $x_i \geq 0$, $i=1,2$. Now we have
$A^TAx_1=A^TAx_2+A^Tb=A^T(Ax_2+b)$.  Now call $b'=Ax_2+b$. Then $Ax=b'$ is unsolvable but $A^TAx=A^Tb'$ is feasible.
However, while the normal equation is always solvable, when nonnegativity constraints are added it may be the case that both $Ax=b, x \geq 0$ and $A^TAx=A^Tb, x \geq 0$ are infeasible.
\end{remark}.

Consider applying the version of TA for solving $Ax=b$ or $A^TAx=A^Tb$ approximately in Algorithm \ref{2.2}.  We consider extension of this to the case where nonnegativity constraint $x \geq 0$ is also added.  There are two issues to consider.  First, since by adding $x \geq 0$ there may be no solution to $A^TAx=A^Tb, x \geq 0$,  there is no stopping rule. This problem can be circumvented by imposing a termination bound. The second issue that need to be addressed is the following important algorithmic aspect of the Triangle Algorithm:
Optimization of a linear function over an ellipsoid, $E_{A,\rho}=\{Ax: \Vert x \Vert \leq \rho\}$ (the main iterative step in the algorithm), can be computed efficiently. This is proved in Theorem \ref{prop0TA}. Specifically, if $c \not =0$,
\begin{equation}  \label{vrend}
{\rm argmax}\{c^Tx : Ax \in E_{A, \rho}\}=\rho \frac{Ac}{\Vert c \Vert}.
\end{equation}
Suppose $\Omega$, if nonempty, is bounded, i.e. there exists $\rho >0$ such that for any $x \in \Omega$,  $\Vert x \Vert \leq \rho$. Consider the problem of maximizing a linear function $c^Tx$ over the intersection of the ellipsoid $E_{A, \rho}$ and the nonnegativity cone, $\{x: x \geq 0\}$. If $c_+$ is the vector that replaces the negative $c_i$'s in $c$ with zero, then it is easy to prove

\begin{prop}
\begin{equation}  \label{vrend2}
{\rm argmax}\{c_+^Tx : Ax \in E_{A, \rho}, \quad x \geq 0\}= \rho \frac{Ac_+}{\Vert c_+ \Vert}.
\end{equation}
\end{prop} 
This means Algorithm \ref{2.2} can be modified  to test if $\Omega$ is nonempty, simply 
by replacing $c$ in the definition of $v_\rho$ by $c_+$ as defined in the above proposition.  The same time complexity bound as that of solving $Ax=b$ applies to these modified algorithms, however need to use a termination upper bound on $\rho$.  

Extension of CTA to LP feasibility is not clear, however there are schemes that can make use of a combination of TA and CTA iterations. We ignore such combination but provide some computational results for testing LP feasibility of both feasible and infeasible cases. 

\subsection{Computational Results for LP feasibility} \label{sec6.1}

We have implemented TA for testing LP feasibility both for feasible and infeasible problems. These are reported in Tables \ref{feas_lp_times} and \ref{ifeas_lp_times}.  Note that the average runtimes of TA for LP feasibility is comparable to our hybrid algorithms for solving rectangular linear systems. These results are promising, and provide additional drive in further research towards LP feasibility.

\begin{table}[!htb]
\centering
\begin{tabular}{|cccc|}
\hline
\multicolumn{4}{|c|}{Average Runtime (in seconds) for Testing LP Feasibility in the Feasible Case}                                   \\ \hline
\multicolumn{1}{|c|}{$n$}      & \multicolumn{1}{c|}{100}  & \multicolumn{1}{c|}{1000} & 10000 \\ \hline
\multicolumn{1}{|c|}{$m$}      & \multicolumn{1}{c|}{200 - 1000}  & \multicolumn{1}{c|}{2000 - 10000} & 20000 - 1000000 \\ \hline
\multicolumn{1}{|c|}{CTA/TA ($\varepsilon = 10^{-8}$)}  & \multicolumn{1}{c|}{1.2} & \multicolumn{1}{c|}{40.17} & 421.28 \\ \hline
\multicolumn{1}{|c|}{CTA/TA ($\varepsilon = 10^{-15}$)}  & \multicolumn{1}{c|}{1.3} & \multicolumn{1}{c|}{42.29} & 430.35 \\ \hline
\multicolumn{1}{|c|}{TA with Non-negativity Constraints} & \multicolumn{1}{c|}{3.2} & \multicolumn{1}{c|}{99.75} & 878.23 \\ \hline
\end{tabular}
\caption{Comparing average runtimes (in seconds) between running our hybrid algorithms and TA with non-negativity constraints for testing LP feasibility in the feasible case.}
\label{feas_lp_times}
\end{table}

\begin{table}[!htb]
\centering
\begin{tabular}{|cccc|}
\hline
\multicolumn{4}{|c|}{Average Runtime (in seconds) for Testing LP Feasibility in the Infeasible Case}                                   \\ \hline
\multicolumn{1}{|c|}{$n$}      & \multicolumn{1}{c|}{100}  & \multicolumn{1}{c|}{1000} & 10000 \\ \hline
\multicolumn{1}{|c|}{$m$}      & \multicolumn{1}{c|}{200 - 1000}  & \multicolumn{1}{c|}{2000 - 10000} & 20000 - 1000000 \\ \hline
\multicolumn{1}{|c|}{CTA/TA ($\varepsilon = 10^{-8}$)}  & \multicolumn{1}{c|}{1.8} & \multicolumn{1}{c|}{52.20} & 481.58 \\ \hline
\multicolumn{1}{|c|}{CTA/TA ($\varepsilon = 10^{-15}$)}  & \multicolumn{1}{c|}{1.9} & \multicolumn{1}{c|}{54.64} & 492.21 \\ \hline
\multicolumn{1}{|c|}{TA with Non-negativity Constraints} & \multicolumn{1}{c|}{4.1} & \multicolumn{1}{c|}{102.96} & 910.40 \\ \hline
\end{tabular}
\caption{Comparing average runtimes (in seconds) between running our hybrid algorithms and TA with non-negativity constraints for testing LP feasibility in the infeasible case.}
\label{ifeas_lp_times}
\end{table}

\section{Concluding Remarks}
 Based on the development of TA and CTA, and their theoretical results from Kalantari \cite{kalantari23}, a summary of which was described in this paper, we  developed an implementation that is robust and easy to use. Using this implementation, we presented extensive computational results from a variety of applications while using many matrices from well-known sources, showing our implementation outperformed two state-of-the-art algorithms GMRES and ``lsqminnorm". We also extended TA for LP feasibility, thus capable of handling linear systems with non-negativity constraints. Interestingly, our computational results show that this extension is comparably fast in relation to solving general linear systems. Therefore, this approach can be integrated in solving linear programming problems, as well as integer programming problems. We remark that we could have also extended this approach and modified Algorithm \ref{2.5} to compute an approximate minimum norm LP feasible solution.  In this article, we also proved some facts on the dynamics of the residuals in the first-order CTA, proving it is essentially equivalent to the dynamics of the algorithm on the diagonal matrix of eigenvalues of $H$.  We also characterized critical regions as lines, dependent on the eigenvalues. Such results on dynamics can be extended to general members of CTA, however it is likely that characterization of corresponding critical regions is more complex.
 
 For simplicity and robustness, in our implementation of CTA we set $H = AA^T$ in Algorithm \ref{1.1}. However, when $A$ is symmetric PSD $H$ can be taken to be $A$. Specialized implementations reduce the number of operations per iteration, also improving the condition number and therefore the theoretical complexity bound. Additionally, as shown in \cite{kalantari23},  the three TA algorithms mentioned in this article each have a corresponding version for when $A$ is symmetric PSD with improved  operation complexity in each iteration. These versions are based on a delicate theoretical connection between TA and CTA analyzed in \cite{kalantari23}. In fact, CTA can be viewed as a dynamic version of TA, where in each iteration the current iterate becomes the center of a new ellipsoid with a particular radius $\rho$.  Although by using $H = AA^T$ we square the condition number of the underlying $H$, based on our extensive computational results, our implementation still outperformed GMRES and ``lsqminnorm". 
 In our implementation and computational results we did not consider preconditioning $A$. Firstly, as is well-known, there is a trade-off in preconditioning  in terms of complexity and runtime \cite{benzisur, benzi3, benzi2, wathenprecon}. Secondly, regardless of preconditioning, the implementations of TA and CTA gave favorable runtimes and sufficiently high-precision approximate solutions. Needless to say, problems for which preconditioning the matrix is amenable, similar approach can be done before running CTA or TA. Nevertheless, future research opportunities can determine the effectiveness of modifications in enhancing and optimizing the implementations of TA and CTA.

\bibliographystyle{plain}
\bibliography{biblio3}
\end{document}